\documentclass[11 pt, draft]{article}
\title{The second rational homology group of the moduli space of curves with level structures}
\author{Andrew Putman}
\usepackage{amsmath}
\usepackage{amssymb}
\usepackage{amsthm}
\usepackage{epsfig}
\usepackage[margin=1.25in]{geometry}
\usepackage{amsfonts}
\usepackage[font=small,format=plain,labelfont=bf,up,textfont=it,up]{caption}
\usepackage{amscd}
\usepackage{pinlabel}
\usepackage{stmaryrd}
\usepackage{type1cm}
\usepackage{calc}
\usepackage{mathptmx}  

\theoremstyle{plain}
\newtheorem{theorem}{Theorem}[section]
\newtheorem{proposition}[theorem]{Proposition}
\newtheorem{lemma}[theorem]{Lemma}
\newtheorem{corollary}[theorem]{Corollary}
\newtheorem{conjecture}[theorem]{Conjecture}
\newtheorem{question}[theorem]{Question}

\theoremstyle{definition}
\newtheorem{definition}[theorem]{Definition}

\theoremstyle{remark}
\newtheorem*{remark}{Remark}
\newtheorem*{example}{Example}

\newtheorem*{warning}{Warning}

\DeclareMathOperator{\Hom}{Hom}

\DeclareMathOperator{\Homeo}{Homeo}
\DeclareMathOperator{\Ker}{ker}

\DeclareMathOperator{\Mod}{Mod}
\newcommand\Torelli{\text{${\mathcal I}$}}
\DeclareMathOperator{\Sp}{Sp}

\DeclareMathOperator{\SL}{SL}

\newcommand\Moduli{\text{${\mathcal M}$}}

\newcommand\Z{\text{$\mathbb{Z}$}}
\newcommand\Q{\text{$\mathbb{Q}$}}

\DeclareMathOperator{\HH}{H}

\DeclareMathOperator{\Max}{max}

\DeclareMathOperator{\Aut}{Aut}

\newcommand\Span[1]{\text{$\langle #1 \rangle$}}
\newcommand\CaptionSpace{\hspace{0.2in}}
\DeclareMathOperator{\Dim}{dim}
\DeclareMathOperator{\Image}{Image}

\newcommand\Figure[3]{
\begin{figure}[t]
\centering
\centerline{\psfig{file=#2,scale=60}}
\caption{#3}
\label{#1}
\end{figure}}

\DeclareMathOperator{\Star}{star}
\DeclareMathOperator{\Link}{link}
\DeclareMathOperator{\Rank}{rk}
\newcommand\CNosep{\mathcal{C}^{\text{ns}}}
\newcommand\Bases{\mathcal{L}}
\DeclareMathOperator{\Picard}{Pic}

\DeclareMathOperator{\One}{\mathbb{I}}

\newcommand\MatTwoTwo[4]{\ensuremath
\left(\begin{smallmatrix}
#1 & #2 \\ #3 & #4
\end{smallmatrix}\right)}

\begin{document}

\maketitle

\begin{abstract}
Let $\Gamma$ be a finite-index subgroup of the mapping
class group of a closed genus $g$ surface that contains the Torelli group.  For
instance, $\Gamma$ can be the level $L$ subgroup or the spin mapping
class group.  We show that $\HH_2(\Gamma;\Q) \cong \Q$ for $g \geq 5$.  A
corollary of this is that the rational Picard groups of the associated
finite covers of the moduli space of curves are equal to $\Q$.  We also prove
analogous results for surface with punctures and boundary components.
\end{abstract} 

\section{Introduction}
\label{section:introduction}

Let $\Sigma_g$ be a closed oriented genus $g$ surface and let $\Mod_g$ be its mapping class group, that is,
the group of isotopy classes of orientation preserving homeomorphisms of $\Sigma_g$ (see 
\cite{FarbMargalitSurvey, IvanovSurvey} for surveys about $\Mod_g$).  Tremendous progress
has been made over the last 40 years in understanding the homology of $\Mod_g$, culminating in
the groundbreaking work of Madsen--Weiss \cite{MadsenWeiss}, who identified $\HH_{\ast}(\Mod_g;\Q)$ in a stable
range.  However, little is known about the homology of finite-index subgroups of $\Mod_g$, or equivalently
about the homology of finite covers of the moduli space of curves.

Denote by $\Torelli_g$ the {\em Torelli group}, that is, the kernel of the representation
$\Mod_g \rightarrow \Sp_{2g}(\Z)$ arising from the action of $\Mod_g$ on $\HH_1(\Sigma_g;\Z)$.
Our main theorem is as follows.  It answers in the affirmative a question of Hain \cite{HainTorelli} which
has since appeared on problem lists of Farb \cite[Conjecture 5.24]{FarbProblems} and
Penner \cite[Problem 11]{PennerProblems}.  

\begin{theorem}[{Rational $\HH_2$ of finite-index subgroups, closed case}]
\label{theorem:mainh2closed}
For $g \geq 5$, let $\Gamma$ be a finite index subgroup of $\Mod_{g}$ such that $\Torelli_{g} < \Gamma$.
Then $\HH_2(\Gamma;\Q) \cong \Q$.
\end{theorem}
\noindent
We also have an analogous result for surfaces with punctures and boundary components; see Theorem \ref{theorem:mainh2} below.

\paragraph{Examples.}
The subgroups of $\Mod_g$ to which Theorem \ref{theorem:mainh2closed} applies are exactly the pullback
to $\Mod_g$ of finite-index subgroups of $\Sp_{2g}(\Z)$.  Two key examples are as follows.

\begin{example}[{Level $L$ subgroup}]
For an integer $L \geq 2$, the {\em level $L$ subgroup} $\Mod_g(L)$ of $\Mod_g$ is the group of mapping
classes that act trivially on $\HH_1(\Sigma_g;\Z/L)$.  The image of $\Mod_g(L)$ in $\Sp_{2g}(\Z)$ is
the kernel of the natural map $\Sp_{2g}(\Z) \rightarrow \Sp_{2g}(\Z/L)$.  This group of matrices, denoted
$\Sp_{2g}(\Z,L)$, is known as the {\em level $L$ subgroup} of $\Sp_{2g}(\Z)$.
\end{example}

\begin{example}[{Spin subgroup}]
Letting $U\Sigma_g$ be the unit tangent bundle of $\Sigma_g$, a {\em spin structure} on $\Sigma_g$
is an element $\omega \in \HH^1(U\Sigma_g;\Z/2)$ such that $\omega(\ell) = 1$, where 
$\ell \in \HH_1(U\Sigma_g;\Z/2)$ is the loop around the fiber.  If $\Sigma_g$ is given
the structure of a Riemann surface, then spin structures on $\Sigma_g$ can be identified
with {\em theta characteristics}, i.e.\ square roots of the canonical bundle \cite[Proposition 3.2]{AtiyahSpin}.
More topologically, Johnson \cite{JohnsonSpin} showed that spin structures on $\Sigma_g$ can
be identified with $\Z/2$-valued quadratic forms $\hat{\omega}$ on $\HH_1(\Sigma_g;\Z/2)$, i.e.\ functions 
$\hat{\omega} : \HH_1(\Sigma_g;\Z/2) \rightarrow \Z/2$ that satisfy
$$\hat{\omega}(x+y) = \hat{\omega}(x) + \hat{\omega}(y) + i(x,y) \quad \text{for all $x,y \in \HH_1(\Sigma_g;\Z/2)$.}$$
Here $i(\cdot,\cdot)$ is the algebraic intersection pairing.  Such quadratic forms are classified up
to isomorphism by their $\Z/2$-valued Arf invariant.  The group $\Mod_g$ acts on the set
of spin structures on $\Sigma_g$, and this action is transitive on the set of spin structures
of a fixed Arf invariant.  If $\omega$ is a spin structure on $\Sigma_g$, then
the stabilizer subgroup $\Mod_g(\omega)$ of $\omega$ in $\Mod_g$ is known as a 
{\em spin mapping class group} (see, e.g., \cite{HarerStableSpin, HarerH2Spin}).
\end{example}

\begin{remark}
The reader should be warned that some papers (e.g. \cite{GalatiusSpin}) use 
the term ``spin mapping class group'' to refer to a certain degree $2$ extension
of $\Mod_g(\omega)$.
\end{remark}

\begin{remark}
If $\omega$ and $\omega'$ are spin structures on $\Sigma_g$ of Arf invariant $0$ and $1$, respectively, then
$\Mod_g(\omega)$ is not isomorphic to $\Mod_g(\omega')$.  Here is a quick proof.  The desired result is trivial
for $g=1$, so we will restrict to the case $g \geq 2$.  It is well-known that there 
are $2^{g-1}(2^g+1)$ (resp.\ $2^{g-1}(2^g-1)$) spin structures on $\Sigma_g$ of Arf invariant $0$ (resp.\ $1$), so
$$[\Mod_g:\Mod_g(\omega)] = 2^{g-1}(2^g+1) \quad \quad \text{and} \quad \quad [\Mod_g:\Mod_g(\omega')] = 2^{g-1}(2^g-1).$$
Ivanov \cite{IvanovComm} proved that if $G,G' < \Mod_g$ are isomorphic 
finite-index subgroups and $g \geq 2$, then there exists some $f \in \Aut(\Mod_g)$ such that
$f(G) = G'$ (Ivanov also proved that for $g \geq 3$, all elements of $\Aut(\Mod_g)$ are
induced by conjugation by elements of the {\em extended mapping class group}, which is
the group of mapping classes that are allowed to reverse orientation; there are ``exotic'' automorphisms
in the case $g=2$).  In particular, if $G,G' < \Mod_g$ are isomorphic finite-index subgroups,
then $[\Mod_g:G] = [\Mod_g:G']$, so $\Mod_g(\omega)$ and $\Mod_g(\omega')$ cannot be isomorphic.
\end{remark}

\paragraph{Moduli space of curves.}
Theorem \ref{theorem:mainh2closed} has consequences for the moduli space $\Moduli_g$ of genus
$g$ Riemann surfaces.  Indeed, $\Mod_g$ is the orbifold fundamental group of $\Moduli_g$, and finite
index subgroups of $\Mod_g$ correspond to finite covers of $\Moduli_g$.  For example, $\Mod_g(L)$ is
the orbifold fundamental group of the moduli space of genus $g$ Riemann surfaces $S$ equipped
with level $L$ structures (i.e.\ distinguished symplectic bases for $\HH_1(S;\Z/L)$).  
Similarly, if $\omega$ is a spin structure of Arf invariant $k$, then $\Mod_g(\omega)$
is the orbifold fundamental group of the moduli space of genus $g$ Riemann surfaces 
equipped with distinguished theta characteristics of parity $k$.

Hain observed in \cite{HainTorelli} that Theorem \ref{theorem:mainh2closed} together
with his computation of the first rational homology group
of these subgroups has the following consequence.

\begin{corollary}[{Picard number one conjecture for moduli spaces of curves with level structures}]
\label{corollary:picard}
For $g \geq 5$, let $\widetilde{\Moduli}_g$ be a finite cover of $\Moduli_g$ whose orbifold fundamental group is
$\Gamma < \Mod_g$.  Assume that $\Torelli_g < \Gamma$.  Then $\Picard(\widetilde{\Moduli}_g) \otimes \Q \cong \Q$.
\end{corollary}

\noindent
For moduli spaces of curves with punctures and boundary components, a similar result follows from
Theorem \ref{theorem:mainh2} below.

\begin{remark}
In a sequel to this paper \cite{PutmanPicardGroup}, the author has computed the {\em integral} Picard groups
and second cohomology groups of $\Mod_g(L)$ for many $L$.  Theorem \ref{theorem:mainh2closed} and
Corollary \ref{corollary:picard} are key ingredients in this calculation.
\end{remark}

\paragraph{Motivation and history.}
What would one expect the second rational homology group of a finite-index subgroup of $\Mod_g$ to be?  

Harer \cite{HarerH2} proved that $\HH_2(\Mod_g;\Q) \cong \Q$ for $g \geq 3$, and via the transfer homomorphism
of group homology one can show that if $A$ is a finite-index subgroup of $B$, then the map 
$\HH_k(A;\Q) \rightarrow \HH_k(B;\Q)$ is
a surjection for all $k$.  Hence the rank of the second rational homology group of our subgroup
is at least $1$, but there is no a priori reason that it cannot be quite large.

However, one of the major themes in the study of $\Mod_g$ is that it shares many properties with
lattices in Lie groups.  Borel (\cite{BorelStability1, BorelStability2}; see also \cite{CharneyCongruence}) 
has proven a very general result about the rational cohomology groups of arithmetic subgroups
of semisimple algebraic groups that implies, for instance,
that for all $k$ there exists some $N$ such that if $\Gamma$ is a finite-index subgroup
of $\SL_n(\Z)$ and $n \geq N$, then $\HH_k(\Gamma;\Q) \cong \HH_k(\SL_n(\Z);\Q)$.

Perhaps inspired by this result, Harer \cite{HarerH2Spin} proved
that the first and second rational homology groups of the spin mapping
class group are the same as those of the whole mapping class group for $g$ sufficiently large.  
Later, Hain (\cite{HainTorelli}; see also \cite{McCarthyCofinite})
proved that if $g \geq 3$ and $\Gamma$ is any finite-index subgroup of $\Mod_g$ that contains $\Torelli_g$, then
$\HH_1(\Gamma;\Q) \cong \HH_1(\Mod_g;\Q) = 0$.  

\begin{remark}
In \cite{Foisy}, Foisy claims to prove Theorem \ref{theorem:mainh2} for the level $2$ subgroup
of $\Mod_g$.  However, Foisy has indicated to us that the proof of Lemma 5.1 in \cite{Foisy} contains
a mistake, so his proof is incomplete.
\end{remark}

\paragraph{Proof sketch.}
We now discuss the proof of Theorem \ref{theorem:mainh2closed}, focusing on the
case of $\Mod_g(L)$.  A first approach to proving Theorem \ref{theorem:mainh2closed} is to attempt to show that
some of the structure of the homology groups of $\Mod_g$ persists in $\Mod_g(L)$.  In particular,
Harer \cite{HarerStability} proved that $\HH_k(\Mod_g;\Z)$ is independent of $g$ for $g$ large.  
Let $\Mod_{g,b}$ denote the mapping
class group of an oriented genus $g$ surface with $b$ boundary components $\Sigma_{g,b}$ (see
\S \ref{section:puncturesboundary} for our conventions on surfaces with boundary).  Any subsurface
inclusion $\Sigma_{g-1,1} \hookrightarrow \Sigma_g$ induces a map $\Mod_{g-1,1} \hookrightarrow \Mod_g$ (``extend
by the identity'').  A more precise statement of part of Harer's theorem is that the induced 
map $\HH_k(\Mod_{g-1,1};\Z) \rightarrow \HH_k(\Mod_g;\Z)$ is an isomorphism for $g$ large.  

We will not need the full strength of Harer's result.  Let $\gamma$ be a nonseparating simple closed
curve on $\Sigma_g$ and let $\Sigma_{g-1,1} \hookrightarrow \Sigma_g$ be a subsurface with 
$\gamma \subset \Sigma_{g} \setminus \Sigma_{g-1,1}$.  Denoting the stabilizer in $\Mod_g$ of the isotopy class of
$\gamma$ by $(\Mod_g)_{\gamma}$, Harer's result implies that for $g$ large, the composition
$$\HH_k(\Mod_{g-1,1};\Z) \longrightarrow \HH_k((\Mod_g)_{\gamma};\Z) \longrightarrow \HH_k(\Mod_g;\Z)$$
is an isomorphism, and hence that the map $\HH_k((\Mod_g)_{\gamma};\Z) \rightarrow \HH_k(\Mod_g;\Z)$ is a surjection.

The key observation underlying the philosophy of our proof is contained in Lemma \ref{lemma:stabilitytrick}
below, which says that to prove Theorem \ref{theorem:mainh2closed}, it is enough to prove
a similar stability result for $\HH_2(\Mod_g(L);\Q)$.  Namely, it is enough to prove the following
theorem.

\begin{theorem}
\label{theorem:stability}
Fix $g \geq 5$ and $L \geq 2$.  Let $\gamma$ be a simple closed nonseparating curve on $\Sigma_g$.  Then
the natural map $\HH_2((\Mod_g(L))_{\gamma};\Q) \rightarrow \HH_2(\Mod_g(L);\Q)$ is a surjection.
\end{theorem}
 
\begin{remark}
For technical reasons, it is {\em not} true that $\HH_2((\Mod_g)_{\gamma};\Q) \cong \HH_2(\Mod_g;\Q)$ for
$g$ large, and similarly for $\Mod_{g}(L)$.  See \S \ref{section:decompthmstrong}.
\end{remark}

\begin{remark}
A naive approach to proving Theorem \ref{theorem:mainh2closed} would be to perform some sort of induction
on $g$.  However, we suspect that Theorem \ref{theorem:mainh2closed} is false for small $g$, so it would be
difficult to establish a base case for the induction.  It is perhaps amusing to observe that our
proof is essentially the inductive step for such an induction, but with no need for a base case!
\end{remark}

We now discuss the proof of Theorem \ref{theorem:stability}.  Analogous homological stability
theorems are known for many sequences of groups (the literature is too large to summarize briefly --
\cite{VanDerKallen} is an important representative of these sorts of results, and \cite[Chapter 2]{KnudsonBook} is
a textbook reference), 
and there has evolved an essentially standard method for proving them.  The basic
idea underlying these proofs goes back to unpublished work of Quillen on the homology groups
of general linear groups over fields.
Given a sequence of groups $\{G_i\}_{i \in \Z}$, one constructs a sequence 
$\{X_i\}_{i \in \Z}$ of highly connected simplicial complexes such that $G_i$ acts on $X_i$.
One then applies a spectral sequence arising from the Borel construction
to decompose $\HH_{\ast}(G_i;\Z)$ in terms of the homology groups of the stabilizer subgroups of simplices in $X_i$.
The $X_i$ are constructed so that these stabilizer subgroups are equal to (or at least
similar to) earlier groups in the sequence of groups.

Harer proved his stability theorem for mapping class groups of closed surfaces by applying an argument of this
type to the {\em nonseparating complex of curves}, denoted $\CNosep_g$.  This is the simplicial
complex whose $(n-1)$-simplices are sets $\{\gamma_1,\ldots,\gamma_n\}$ of isotopy class
of nonseparating simple closed curves on $\Sigma_g$ which can be realized disjointly with
$\Sigma_g \setminus (\gamma_1 \cup \cdots \cup \gamma_n)$ connected.
We attempt to imitate this.  

Alas, it does not quite work.  The problem is that
for $H_2$, the output of the homological stability machinery is that the natural map
\begin{equation}
\label{eqn:stability}
\bigoplus_{\gamma \in (\CNosep_{g} / \Mod_{g}(L))^{(0)}} \HH_2((\Mod_{g}(L))_{\tilde{\gamma}};\Q) \longrightarrow \HH_2(\Mod_{g}(L);\Q)
\end{equation}
is surjective for $g \geq 5$.  Here $\tilde{\gamma}$ is an arbitrary lift of $\gamma$ to $\CNosep_g$ and
$(\Mod_{g}(L))_{\tilde{\gamma}}$ is the stabilizer subgroup of $\tilde{\gamma}$ in $\Mod_{g}(L)$.  Two
different lifts give conjugate stabilizer subgroups and hence the same homology groups.  
This is not enough because unlike $\Mod_g$, the group $\Mod_{g}(L)$ does not act transitively on the set of nonseparating
simple closed curves.  To solve this problem, we must 
perform a group cohomological computation to show that the stabilizers of any
two simple closed nonseparating curves give the same ``chunk'' of homology, and thus that 
the map $\HH_2((\Mod_g(L))_{\gamma};\Q) \rightarrow \HH_2(\Mod_g(L);\Q)$ is surjective.  The key
to this computation is a certain vanishing result of the author (\cite{PutmanVanishing}; see
Lemma \ref{lemma:vanishing} below) for the twisted first homology groups
of $\Mod_{g}(L)$ with coefficients in the homology groups of abelian covers of the surface.

\begin{remark}
One reason why the standard homological stability machinery must break when applied to $\Mod_g(L)$ is that
if it worked, then it would yield an {\em integral} homology stability result.  However, it is known
that even $\HH_1(\Mod_g(L);\Z)$ does not stabilize (see \cite{PutmanPicardGroup}).
\end{remark}

\paragraph{Conjectures.}
Given Theorem \ref{theorem:mainh2closed} together with Hain's theorem about the first
rational homology groups of $\Mod_g(L)$, we conjecture that a similar result must
hold for the higher homology groups.  More specifically, we make the following conjecture.

\begin{conjecture}[{Isomorphism conjecture}]
\label{conjecture:isomorphism}
For fixed $L \geq 2$ and $k \geq 1$, there exists some $N$ such that if $g \geq N$, then
the map $\HH_k(\Mod_g(L);\Q) \rightarrow \HH_k(\Mod_g;\Q)$ is an isomorphism.
\end{conjecture}

\noindent
Just as in the proof of Theorem \ref{theorem:mainh2closed}, to prove Conjecture
\ref{conjecture:isomorphism}, it is enough to prove Conjecture \ref{conjecture:weakstability}
below.

\begin{conjecture}[{Stability conjecture}]
\label{conjecture:weakstability}
For fixed $L \geq 2$ and $k \geq 1$, there exists some $N$ such that if $g \geq N$ and
$\gamma$ is a simple closed nonseparating curve on $\Sigma_g$, then
the map $\HH_k((\Mod_{g}(L))_{\gamma};\Q) \rightarrow \HH_k(\Mod_{g}(L);\Q)$
is a surjection.
\end{conjecture}

\noindent
Moreover, a similar argument establishes an appropriate analogue of \eqref{eqn:stability}.  However,
it seems difficult to perform the necessary calculation to go from there to Conjecture
\ref{conjecture:weakstability}.

Conjecture \ref{conjecture:isomorphism} is also related to the homology of $\Torelli_g$.  In particular,
we have the following folk conjecture/question.

\begin{question}[{Homology of Torelli}]
\label{question:torelli}
For a fixed $k'$, does there exists some $N$ such that if $g \geq N$, then $\HH_{k'}(\Torelli_g;\Q)$ is
a finite-dimensional vector space and the action of $\Sp_{2g}(\Z)$ on $\HH_{k'}(\Torelli_g;\Q)$
arising from the conjugation action of $\Mod_g$ on $\Torelli_g$ extends to a rational representation
of the algebraic group $\Sp_{2g}$?
\end{question}

Johnson \cite{JohnsonAbel} proved that the answer to Question \ref{question:torelli} is yes for $k'=1$.  
We claim that if Question \ref{question:torelli} had an affirmative answer for all $k' \leq k$, then 
it would provide a quick proof of Conjecture \ref{conjecture:isomorphism} for $k$.  
Thus Theorem \ref{theorem:mainh2} provides some additional evidence for an affirmative answer to
Question \ref{question:torelli}.

A sketch of the proof of this claim is as follows.  
We have a commutative diagram of short exact sequences
$$\begin{CD}
1 @>>> \Torelli_g @>>> \Mod_g(L) @>>> \Sp_{2g}(\Z,L) @>>> 1\\
@.     @VVV            @VVV           @VVV                @.\\
1 @>>> \Torelli_g @>>> \Mod_g    @>>> \Sp_{2g}(\Z)   @>>> 1
\end{CD}$$
This induces a map of the associated homology Hochschild-Serre spectral sequences.  
The Borel stability theorem mentioned above \cite{BorelStability1, BorelStability2} 
applies to the twisted
coefficient systems provided by Question \ref{question:torelli}.
Moreover, the range of stability is independent of the coefficient system.  
We conclude that our map of Hochschild-Serre spectral sequences is an isomorphism
on all terms in the $(k+1) \times k$ lower left hand corners of
the $E_2$ pages of these spectral sequences.  Zeeman's spectral sequence comparison theorem \cite{ZeemanSpec} then gives
the desired result.

\begin{remark}
The Borel stability theorem for twisted coefficient systems
is usually stated cohomologically, but since we are working over $\Q$
it has a dual homological version; see Theorem \ref{theorem:duality} below.
\end{remark}

\paragraph{Outline.}
We begin in \S \ref{section:preliminaries} with various preliminaries.  These include
a discussion of surfaces with punctures and boundary components,
some results about group homology with twisted coefficient systems, some facts about simplicial
complexes and combinatorial manifolds, and the Birman exact sequence.  We also
state Theorem \ref{theorem:mainh2}, which is a generalization of Theorem \ref{theorem:mainh2closed}
to surfaces with punctures and boundary components.  Next, in \S \ref{section:firsthomology}
we prove several theorems about the first homology groups of $\Mod_g(L)$ with various coefficient systems.
In \S \ref{section:closedlevel} we begin the proof of Theorem \ref{theorem:mainh2} by reducing the theorem
to the special case of the level $L$ subgroup of $\Mod_g$.  The next step is in \S \ref{section:stabilitytrick},
where we prove the equivalence of Theorems \ref{theorem:mainh2} and \ref{theorem:stability} (and more generally,
Conjectures \ref{conjecture:isomorphism} and \ref{conjecture:weakstability}).
We then give our proof that $\HH_2(\Mod_g(L);\Q)$ stabilizes in an appropriate sense, establishing
our main theorem.  This proof depends on two lemmas, which we prove in \S \ref{section:decompthm} --
\S \ref{section:weakstability}.

\paragraph{Notation and conventions.}
Throughout this paper, we will systematically confuse simple closed curves with their homotopy
classes.  Hence (based/unbased) curves are said to be simple closed curves if they are (based/unbased)
homotopic to simple closed curves, etc.  For $x,y \in \HH_1(\Sigma_{g,b};\Z/L)$, we will
denote by $i(x,y) \in \Z/L$ the algebraic intersection number of $x$ and $y$.  Also, for a
simple closed curve $\gamma$ we will denote by $T_{\gamma}$ the right Dehn twist about $\gamma$.

\paragraph{Acknowledgments.}
I wish to thank Mladen Bestvina, Ruth Charney, Pokman Cheung, Benson Farb, Joel Foisy, Richard Hain, Dan Margalit,
Robert Penner, and Karen Vogtmann for useful conversations and comments.

\section{Preliminaries}
\label{section:preliminaries}

\subsection{Surfaces with punctures and boundary components}
\label{section:puncturesboundary}

Let $\Sigma_{g,b}^n$ denote an oriented genus $g$ surface with $b$ boundary components and $n$ punctures.  Observe
that a homeomorphism of $\Sigma_{g,b}^n$ extends over the punctures in a natural way.
Define
\begin{align*}
\Mod_{g,b}^n = \pi_0(\{\text{$f \in \Homeo^{+}(\Sigma_{g,b}^n)$ $|$ }&\text{$f$ restricts to the identity on the boundary}\\
                                                                     &\text{and does not permute the punctures}\}).
\end{align*}
If $S$ is a surface such that $S \cong \Sigma_{g,b}^n$, then we will sometimes
write $\Mod(S)$ instead of $\Mod_{g,b}^n$.

Filling in the punctures and gluing discs to the boundary components induces an embedding
$\Sigma_{g,b}^n \hookrightarrow \Sigma_g$, and by extending elements of $\Mod_{g,b}^n$ by
the identity we get an induced map $\Mod_{g,b}^n \rightarrow \Mod_g$.  This
yields a canonical action of $\Mod_{g,b}^n$ on $\HH_1(\Sigma_g;R)$ for any ring $R$.  Define
\begin{align*}
\Torelli_{g,b}^n = &\{\text{$f \in \Mod_{g,b}^n$ $|$ $f$ acts trivially on $\HH_1(\Sigma_g;\Z)$}\},\\
\Mod_{g,b}^n(L) = &\{\text{$f \in \Mod_{g,b}^n$ $|$ $f$ acts trivially on $\HH_1(\Sigma_g;\Z/L)$}\}.
\end{align*}
If $S$ is a surface such that $S \cong \Sigma_{g,b}^n$, then we will sometimes
write $\Torelli(S)$ and $\Mod(S,L)$ to denote the subgroups $\Torelli_{g,b}^n$ and $\Mod_{g,b}^n(L)$
of $\Mod(S) \cong \Mod_{g,b}^n$.

\begin{remark}
We will frequently omit the $b$ or the $n$ from our notation if they vanish.
\end{remark}

We can now state a more general version of our main theorem.

\begin{theorem}[{Rational $\HH_2$ of finite-index subgroups, general case}]
\label{theorem:mainh2}
For $g \geq 5$ and $b,n \geq 0$, let $\Gamma$ be a finite index subgroup of $\Mod_{g,b}^n$ such that $\Torelli_{g,b}^n < \Gamma$.
Then $\HH_2(\Gamma;\Q) \cong \HH_2(\Mod_{g,b}^n;\Q) \cong \Q^{n+1}$.
\end{theorem}

\subsection{Notation and basic facts about group homology}

We will make extensive use of group homology with twisted coefficient systems.  We now remind
the reader of the twisted analogues of several standard tool in untwisted group homology.
A basic reference is \cite{BrownCohomology}.

\paragraph{Coinvariants.}
Let $G$ be a group and $M$ be a $G$-module.  The module of coinvariants of this action, denoted
$M_G$, is the quotient of $M$ by the submodule generated by the set $\{\text{$m - g(m)$ $|$ $g \in G$, $m \in M$}\}$.
We have $\HH_0(G;M) \cong M_G$.

\paragraph{Long exact sequence.}
Let $G$ be a group and
$$0 \longrightarrow M_1 \longrightarrow M_2 \longrightarrow M_3 \longrightarrow 0$$
be a short exact sequence of $G$-modules.  There is then a long exact sequence
$$\cdots \longrightarrow \HH_{k+1}(G;M_3) \longrightarrow \HH_k(G;M_1) \longrightarrow \HH_k(G;M_2) \longrightarrow \HH_k(G;M_3) \longrightarrow \cdots.$$

\paragraph{The Hochschild-Serre spectral sequence.}
Let 
$$1 \longrightarrow K \longrightarrow G \longrightarrow Q \longrightarrow 1$$
be a short exact sequence of groups and let $M$ be a $G$-module.  The homology Hochschild-Serre spectral
sequence is a first quadrant spectral sequence converging to $\HH_{\ast}(G;M)$ whose $E^2$ page is of the
form
$$E^2_{p,q} = \HH_p(Q;\HH_q(K;M)).$$
If our short exact sequence of groups is split and $M$ is a trivial $G$-module, then all the differentials coming out of
the bottom row of this spectral sequence vanish.  Also, 
this spectral sequence is natural in the sense that a morphism of short exact sequences
induces a morphism of spectral sequences.  The edge groups have the following interpretations.
\begin{itemize}
\item $E^{\infty}_{p,0}$, a subgroup of $E^2_{p,0} \cong \HH_p(Q;\HH_0(K;M)) \cong \HH_p(Q;M_K)$, is
equal to 
$$\Image(\HH_p(G;M) \rightarrow \HH_p(Q;M_K)).$$
\item $E^{\infty}_{0,q}$, a quotient of $E^2_{0,q} \cong \HH_0(Q;\HH_q(K;M)) \cong (\HH_q(K;M))_Q$, is
isomorphic to 
$$\Image(\HH_q(K;M) \rightarrow \HH_q(G;M)).$$
\end{itemize}
A standard consequence of this spectral sequence is the natural 5-term exact sequence
$$\HH_2(G;M) \longrightarrow \HH_2(Q;M_K) \longrightarrow (\HH_1(K;M))_Q \longrightarrow \HH_1(G;M) \longrightarrow \HH_1(Q;M_K) \longrightarrow 0.$$
A similar spectral sequence exists in group cohomology.

\paragraph{The Gysin sequence.}
Let
\begin{equation}
\label{eqn:centralz}
1 \longrightarrow \Z \longrightarrow \Gamma \longrightarrow G \longrightarrow 1
\end{equation}
be a central extension of groups and let $R$ be a ring.  A standard consequence of the Hochschild-Serre spectral sequence
for \eqref{eqn:centralz} is the {\em Gysin sequence}, which is a natural exact sequence of the form
$$\cdots \longrightarrow \HH_{n-1}(G;R) \longrightarrow \HH_n(\Gamma;R) \longrightarrow \HH_n(G;R) \longrightarrow \HH_{n-2}(G;R) \longrightarrow \HH_{n-1}(\Gamma;R) \longrightarrow \cdots$$

\paragraph{Duality.}
We have the following duality between homology and cohomology.

\begin{theorem}
\label{theorem:duality}
Let $G$ be a group and let $M$ be a $G$-vector space over $\Q$.  Define $M' = \Hom(M,\Q)$.  Then for every
$k \geq 0$, there is a natural isomorphism $\HH^k(G;M') \cong \Hom(\HH_k(G;M),\Q)$.
\end{theorem}

\begin{remark}
We do not know of a reference for this result, but the proof is essentially
identical to the proof of \cite[Proposition 7.1]{BrownCohomology}.
\end{remark}

\subsection{Simplicial complexes and combinatorial manifolds}
\label{section:simpcpx}

Our basic reference for simplicial complexes is \cite[Chapter 3]{Spanier}.  Let us recall the definition
of a simplicial complex given there.

\begin{definition}
A {\em simplicial complex} $X$ is a set of nonempty finite sets (called {\em simplices}) such that
if $\Delta \in X$ and $\emptyset \neq \Delta' \subset \Delta$, then $\Delta' \in X$.
The {\em dimension} of a simplex $\Delta \in X$ is $|\Delta|-1$ and is denoted $\Dim(\Delta)$.
For $k \geq 0$, the subcomplex of $X$ consisting of all simplices
of dimension at most $k$ (known as the {\em $k$-skeleton of $X$}) will be denoted $X^{(k)}$.
If $X$ and $Y$ are simplicial complexes, then a {\em simplicial map} from $X$ to $Y$ is a function
$f : X^{(0)} \rightarrow Y^{(0)}$ such that if $\Delta \in X$, then $f(\Delta) \in Y$.
\end{definition}

If $X$ is a simplicial complex, then we will define
the geometric realization $|X|$ of $X$ in the standard way (see \cite[Chapter 3]{Spanier}).  When
we say that $X$ has some topological property (e.g.\ simple-connectivity), we will mean that $|X|$ possesses
that property.

Next, we will need the following definitions.

\begin{definition}
Consider a simplex $\Delta$ of a simplicial complex $X$.
\begin{itemize}
\item The {\em star} of $\Delta$ (denoted $\Star_X(\Delta)$) is the subcomplex of $X$ consisting
of all $\Delta' \in X$ such that there is some $\Delta'' \in X$ with $\Delta,\Delta' \subset \Delta''$.
By convention, we will also define $\Star_X(\emptyset) = X$.
\item The {\em link} of $\Delta$ (denoted $\Link_X(\Delta)$) is the subcomplex of $\Star_X(\Delta)$
consisting of all simplices that do not intersect $\Delta$.  By convention, we will also define
$\Link_X(\emptyset) = X$.
\end{itemize}
\end{definition}

For $n \leq -1$, we will say that the empty set is both an $n$-sphere and a closed $n$-ball.  Also, if $X$
is a space then we will say that $\pi_{-1}(X)=0$ if $X$ is nonempty and that $\pi_{k}(X)=0$ for all $k \leq -2$.
With these conventions, it is true for all $n \in \Z$ that
a space $X$ satisfies $\pi_n(X)=0$ if and only if every map of an $n$-sphere into $X$ can be extended to a map
of a closed $(n+1)$-ball into $X$.

Finally, we will need the following definition.  A basic reference is \cite{RourkeSanderson}.

\begin{definition}
For $n \geq 0$, a {\em combinatorial $n$-manifold} $M$ is a nonempty simplicial complex that
satisfies the following inductive property.  If $\Delta \in M$, then $\Dim(\Delta) \leq n$.
Additionally, if $n-\Dim(\Delta)-1 \geq 0$, then $\Link_M(\Delta)$ is a combinatorial $(n-\Dim(\Delta)-1)$-manifold
homeomorphic to either an $(n-\Dim(\Delta)-1)$-sphere or a closed $(n-\Dim(\Delta)-1)$-ball.  We will denote by
$\partial M$ the subcomplex of $M$ consisting of all simplices $\Delta$ such that $\Dim(\Delta) < n$ and such that
$\Link_M(\Delta)$ is homeomorphic to a closed $(n-\Dim(\Delta)-1)$-ball.
If $\partial M = \emptyset$ then $M$  is said to be {\em closed}.  A
combinatorial $n$-manifold homeomorphic to an $n$-sphere (resp. a closed $n$-ball) will be called
a {\em combinatorial $n$-sphere} (resp. a {\em combinatorial $n$-ball}).
\end{definition}

It is well-known that if $\partial M \neq \emptyset$, then $\partial M$ is a closed combinatorial $(n-1)$-manifold
and that if $B$ is a combinatorial $n$-ball, then $\partial B$ is a combinatorial $(n-1)$-sphere.

\begin{warning}
There exist simplicial complexes that are homeomorphic to manifolds but are {\em not} combinatorial
manifolds.
\end{warning}

The following is an immediate consequence of the Zeeman's extension \cite{ZeemanSimp} of
the simplicial approximation theorem.

\begin{lemma}
\label{lemma:simpapprox}
Let $X$ be a simplicial complex and $n \geq 0$.  The following hold.
\begin{enumerate}
\item Every element of $\pi_n(X)$ is represented by a simplicial map $S \rightarrow X$,
where $S$ is a combinatorial $n$-sphere.
\item If $S$ is a combinatorial $n$-sphere and $f : S \rightarrow X$ is a nullhomotopic
simplicial map, then there is a combinatorial
$(n+1)$-ball $B$ with $\partial B = S$ and a simplicial map $g : B \rightarrow X$ such that $g|_{S} = f$.
\end{enumerate}
\end{lemma}

\subsection{The Birman exact sequence and stabilizers of simple closed curves}
\label{section:birmanexactsequence}

We will make heavy use of two analogues for $\Mod_{g,b}^n(L)$ of the classical
Birman exact sequence.  We start by recalling the statement of one version of the
classical case.

\begin{theorem}[{Johnson, \cite{JohnsonFinite}}]
\label{theorem:birmanexactsequence}
Fix $g \geq 2$ and $b,n \geq 0$.  Gluing a disc to a boundary
component $\beta$ of $\Sigma_{g,b+1}^n$ induces an exact sequence
$$1 \longrightarrow \pi_1(U\Sigma_{g,b}^n) \longrightarrow \Mod_{g,b+1}^n \longrightarrow \Mod_{g,b}^n \longrightarrow 1,$$
where $U\Sigma_{g,b}^n$ is the unit tangent bundle of $\Sigma_{g,b}^n$.  If $b \geq 1$, then this sequence
splits via a map $\Mod_{g,b}^n \rightarrow \Mod_{g,b+1}^n$ induced by an embedding 
$\Sigma_{g,b}^n \hookrightarrow \Sigma_{g,b+1}^n$ such that 
$\overline{\Sigma_{g,b+1}^n \setminus \Sigma_{g,b}^n}$ is homeomorphic to $\Sigma_{0,3}$ and contains $\beta$
(see Figure \ref{figure:birman}.a).  
Finally, in all cases we have an inclusion $\pi_1(U\Sigma_{g,b}^n) < \Torelli_{g,b+1}^n$.
\end{theorem}

\begin{remark}
The mapping class associated to an element of $\pi_1(U\Sigma_{g,b}^n)$ ``drags'' the boundary
component along a path while allowing it to rotate.
\end{remark}

Since in Theorem \ref{theorem:birmanexactsequence} we have $\pi_1(U\Sigma_{g,b}^n) < \Torelli_{g,b+1}^n < \Mod_{g,b+1}^n(L)$,
the following is an immediate corollary.

\begin{corollary}
\label{corollary:birmanexactsequencel1}
Fix $g \geq 2$, $L \geq 2$, and $b,n \geq 0$.  Gluing a disc to a boundary
component of $\Sigma_{g,b+1}^n$ induces an exact sequence
$$1 \longrightarrow \pi_1(U\Sigma_{g,b}^n) \longrightarrow \Mod_{g,b+1}^n(L) \longrightarrow \Mod_{g,b}^n(L) \longrightarrow 1,$$
where $U\Sigma_{g,b}^n$ is the unit tangent bundle of $\Sigma_{g,b}^n$.  If $b \geq 1$, then this sequence
splits.
\end{corollary}

\Figure{figure:birman}{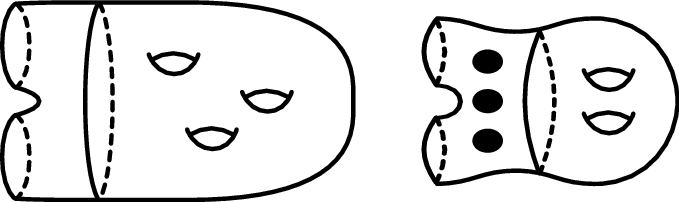}{a. The Birman exact sequence for $\Mod_{3,2}$ splits via the map
$\Mod_{3,1} \rightarrow \Mod_{3,2}$ induced by the indicated inclusion map 
$\Sigma_{3,1} \hookrightarrow \Sigma_{3,2}$.
\CaptionSpace b. $\Sigma_{2,1}$ embedded in $\Sigma_{2,2}^3$ as in the proof
of Lemma \ref{lemma:huhcoinv}}

The other version of the Birman exact sequence we need will help us describe the stabilizer in $\Mod_g(L)$
of simple closed nonseparating curves (there are analogous results for surfaces with boundary, but we
will only need the closed case so we will not state them).  First, some definitions.  For later use,
we will make them a bit more general than needed for our theorem.

\begin{definition}
\label{definition:curvestab}
Let $\gamma_1,\ldots,\gamma_k$ be disjoint simple closed curves on $\Sigma_{g,b}$ such that
$\Sigma_{g,b} \setminus (\gamma_1 \cup \cdots \cup \gamma_k)$ is connected and let 
$N$ be an open regular neighborhood of
$\gamma_1 \cup \cdots \cup \gamma_k$.
Define $\Sigma_{g,b,\gamma_1,\ldots,\gamma_k} = \Sigma_g \setminus N$ and
$\Mod_{g,b,\gamma_1,\ldots,\gamma_k} = \Mod(\Sigma_{g,b,\gamma_1,\ldots,\gamma_k})$.  
Observe that $\Sigma_{g,b,\gamma_1,\ldots,\gamma_k} \cong \Sigma_{g-k,2k+b}$ and 
$\Mod_{g,b,\gamma_1,\ldots,\gamma_k} \cong \Mod_{g-k,2k+b}$.  Next,
let $i : \Mod_{g,b,\gamma_1,\ldots,\gamma_k} \rightarrow \Mod_{g,b}$ be the map induced by the inclusion 
$\Sigma_{g,b,\gamma_1,\ldots,\gamma_k} \hookrightarrow \Sigma_{g,b}$.  We then define
\begin{align*}
\Torelli_{g,b,\gamma_1,\ldots,\gamma_k} &:= i^{-1}(\Torelli_{g,b}),\\
\Mod_{g,b,\gamma_1,\ldots,\gamma_k}(L) &:= i^{-1}(\Mod_{g,b}(L)) \ \ \ \ \ (L \geq 2).
\end{align*}
We will frequently omit the $b$ when it equals $0$.
\end{definition}

\begin{remark}
Observe that $\Torelli_{g,b,\gamma_1,\ldots,\gamma_k} \ncong \Torelli_{g-k,2k+b}$ and 
$\Mod_{g,b,\gamma_1,\ldots,\gamma_k}(L) \ncong \Mod_{g-k,2k+b}(L)$ for $L \geq 2$.  
\end{remark}

A theorem like Corollary \ref{corollary:birmanexactsequencel1} cannot be true as stated for
$\Torelli_{g,\gamma}$ or $\Mod_{g,\gamma}(L)$ since 
the twist about a boundary component is not in $\Torelli_{g,\gamma}$ or
$\Mod_{g,\gamma}(L)$.  In \cite[Theorem 4.1]{PutmanCutPaste}, however, the author proved
an analogue of Corollary \ref{corollary:birmanexactsequencel1} for $\Torelli_{g,\gamma}$.  We will
need a version of this for $\Mod_{g,\gamma}(L)$.  

Let $\gamma$ be a nonseparating simple closed curve on $\Sigma_g$.  Since
$\Mod_{g,\gamma} \cong \Mod_{g-1,2}$, Theorem \ref{theorem:birmanexactsequence} gives a split 
short exact sequence
\begin{equation}
\label{eqn:modgammaexactseq}
1 \longrightarrow \pi_1(U\Sigma_{g-1,1}) \longrightarrow \Mod_{g,\gamma} \longrightarrow \Mod_{g-1,1} \longrightarrow 1.
\end{equation}
It is easy to see that the map $\Mod_{g,\gamma} \rightarrow \Mod_{g-1,1}$ restricts to a surjection 
$\Mod_{g,\gamma}(L) \rightarrow \Mod_{g-1,1}(L)$.  We then have the following theorem.  Recall
that if $x$ is a simple closed curve on $\Sigma_g$, then $T_x \in \Mod_g$ denotes the right Dehn twist about $x$.
Observe that $T_x^L \in \Mod_g(L)$.

\begin{theorem}
\label{theorem:birmanexactsequencel2}
For $g \geq 3$ and $L \geq 2$, let $\gamma$ be a nonseparating simple closed curve on $\Sigma_{g}$.  Then gluing a disc
to a boundary component $\beta$ of $\Sigma_{g,\gamma}$ induces
a split exact sequence
$$1 \longrightarrow \overline{K}_{g-1,1} \longrightarrow \Mod_{g,\gamma}(L) \longrightarrow \Mod_{g-1,1}(L) \longrightarrow 1.$$
Here $\overline{K}_{g-1,1}$ is a subgroup of $\pi_1(U\Sigma_{g-1,1})$ that fits into a split exact sequence
$$1 \longrightarrow \Z \longrightarrow \overline{K}_{g-1,1} \longrightarrow K_{g-1,1} \longrightarrow 1,$$
where $\Z$ is generated by $T_{\beta}^L$ and 
$K_{g-1,1}$ is the kernel of the map $\pi_1(\Sigma_{g-1,1}) \rightarrow \HH_1(\Sigma_{g-1,1};\Z/L)$.
\end{theorem}
\begin{proof}
Restricting exact sequence \eqref{eqn:modgammaexactseq} to $\Mod_{g,\gamma}(L) < \Mod_{g,\gamma}$,
we get a split exact sequence
$$1 \longrightarrow \overline{K}_{g-1,1} \longrightarrow \Mod_{g,\gamma}(L) \longrightarrow Q \longrightarrow 1,$$
where $Q = \Image(\Mod_{g,\gamma}(L) \rightarrow \Mod_{g-1,1}) \cong \Mod_{g-1,1}(L)$ and 
$\overline{K}_{g-1,1} = \pi_1(U\Sigma_{g-1,2}) \cap \Mod_{g,\gamma}(L)$.  We must prove the 
indicated characterization of $\overline{K}_{g-1,1} < \pi_1(U\Sigma_{g-1,1})$.

We have another split exact sequence
$$1 \longrightarrow \Z \longrightarrow \pi_1(U\Sigma_{g-1,1}) \longrightarrow \pi_1(\Sigma_{g-1,1}) \longrightarrow 1,$$
where $\Z$ is generated by $T_{\beta}$ (we remark that the splitting of this exact sequence is not natural -- from
a group theoretic point of view, its existence simply follows from the fact that $\pi_1(\Sigma_{g-1,1})$ is free).  
Since $\overline{K}_{g-1,1} \cap \Span{T_\beta} = \Span{T_{\beta}^L}$,
we can restrict this exact sequence to $\overline{K}_{g-1,1}$ and get a split
exact sequence
$$1 \longrightarrow \Z \longrightarrow \overline{K}_{g-1,1} \longrightarrow K_{g-1,1} \longrightarrow 1,$$
where $\Z$ is generated by $T_{\beta}^L$ and $K_{g-1,1} = \Image(\overline{K}_{g-1,1} \rightarrow \pi_1(\Sigma_{g-1,1}))$.
Our goal is to show that $K_{g-1,1} = \Ker(\pi_1(\Sigma_{g-1,1}) \rightarrow \HH_1(\Sigma_{g-1,1};\Z/L))$.  The proof
of this is similar to the proof of \cite[Theorem 4.1]{PutmanCutPaste}, and is thus omitted.
\end{proof}

\section{The first homology groups}
\label{section:firsthomology}
In this section, we generalize work of Hain \cite{HainTorelli} to calculate the first homology
groups of $\Mod_{g,b}^n(L)$ and related groups with various systems of coefficients.

\subsection{Hain's results}

We will need the following two special cases of a theorem of Hain concerning the first homology groups of $\Mod_{g,b}^n(L)$.
To simplify our notation, we will denote $\Mod_{g,b}^n$ by $\Mod_{g,b}^n(1)$. 

\begin{theorem}[{Hain, \cite[Prop.\ 5.2]{HainTorelli}}]
\label{theorem:hainh1}
For $L \geq 1$, $g \geq 3$ and $b,n \geq 0$, we have $\HH_1(\Mod_{g,b}^n(L);\Q) = 0$.
\end{theorem}

\begin{theorem}[{Hain, \cite[Prop.\ 5.2]{HainTorelli}}]
\label{theorem:hainh1h}
For $L \geq 1$, $g \geq 3$ and $b,n \geq 0$, we have 
$$\HH_1(\Mod_{g,b}^n(L);\HH_1(\Sigma_g;\Q)) \cong \Q^{b+n}.$$
\end{theorem}

\begin{remark}
Hain proves Theorem \ref{theorem:hainh1h} for cohomology, but the indicated theorems for homology follow
from Theorem \ref{theorem:duality} and the fact that we have an isomorphism 
$\Hom(\HH_1(\Sigma_{g};\Q),\Q) \cong \HH_1(\Sigma_{g};\Q)$ of $\Mod_{g,b}^n$-modules arising
from the algebraic intersection form.
\end{remark}

\subsection{Unit tangent bundle coefficients}

To goal of this section is Lemma \ref{lemma:h1modluh} below, which says that
$\HH_1(\Mod_{g,b}^n(L);\HH_1(U\Sigma_{g,b}^n;\Q))=0$.  We will need three preliminary lemmas.

\begin{lemma}
\label{lemma:h1uhclosed}
For $g \geq 2$, we have $\HH_1(U\Sigma_{g};\Q) \cong \HH_1(\Sigma_g;\Q)$ and $\HH^1(U\Sigma_g;\Q) \cong \HH^1(\Sigma_g;\Q)$.
\end{lemma}
\begin{proof}
An immediate consequence of the standard group presentation
\begin{equation*}
\pi_1(U\Sigma_g) \cong \langle \text{$a_1, b_1, \ldots, a_g, b_g, t$ $|$ $[a_1,b_1] \cdots [a_g,b_g] = t^{2-2g}$} \rangle. \qedhere
\end{equation*}
\end{proof}

\begin{lemma}
\label{lemma:huhcoinv}
For $L \geq 1$, $g \geq 2$, and $b,n \geq 0$, we have
$$(\HH_1(U\Sigma_{g,b}^n;\Q))_{\Mod_{g,b}^n(L)} = (\HH_1(\Sigma_{g,b}^n;\Q))_{\Mod_{g,b}^n(L)} = 0.$$
\end{lemma}
\begin{proof}
We first prove that $(\HH_1(\Sigma_{g,b}^n;\Q))_{\Mod_{g,b}^n(L)} = 0$.
The group $\HH_1(\Sigma_{g,b}^n;\Q)$ is generated by the homology classes of oriented simple
closed nonseparating curves (this is true even if $b$ or $n$ are nonzero!).  Let $a$ be such a curve, and
let $b$ an oriented simple closed curve that intersects $a$ once.  We then have $T_{a}^L([b]) = [b] \pm L [a]$,
so in $(\HH_1(\Sigma_{g,b}^n;\Q))_{\Mod_{g,b}^n(L)}$ we have $[b]$ equal to $[b] \pm L [a]$; i.e.\ $L [a] = 0$, as
desired.

We now prove that $(\HH_1(U\Sigma_{g,b}^n;\Q))_{\Mod_{g,b}^n(L)} = (\HH_1(\Sigma_{g,b}^n;\Q))_{\Mod_{g,b}^n(L)}$.  For
$b=n=0$, this follows from Lemma \ref{lemma:h1uhclosed}.  Otherwise, $U\Sigma_{g,b}^n$ is a trivial $S^1$-bundle, so we
have a short exact sequence
$$0 \longrightarrow \Q \longrightarrow \HH_1(U\Sigma_{g,b}^n;\Q) \longrightarrow \HH_1(\Sigma_{g,b}^n;\Q) \longrightarrow 0.$$
The kernel $\Q$ is generated by the homology class of the fiber.  It is enough to show that the kernel $\Q$ of this 
exact sequence is killed when we pass to the coinvariants of $\Torelli_{g,b}^n < \Mod_{g,b}^n(L)$ 
acting on $\HH_1(U\Sigma_{g,b}^n;\Q)$.
Observe (see Figure \ref{figure:birman}.b) 
that there is an embedding $\Sigma_{g,1} \hookrightarrow \Sigma_{g,b}^n$ 
that induces a commutative diagram of short exact sequences
$$\begin{CD}
0 @>>> \Q          @>>> \HH_1(U\Sigma_{g,1};\Q)   @>>> \HH_1(\Sigma_{g,1};\Q)   @>>> 0 \\
@.     @VV{\cong}V      @VVV                          @VVV                          @.\\
0 @>>> \Q          @>>> \HH_1(U\Sigma_{g,b}^n;\Q) @>>> \HH_1(\Sigma_{g,b}^n;\Q) @>>> 0
\end{CD}$$
It follows that it is enough to prove that $(\HH_1(U\Sigma_{g,1};\Q))_{\Torelli_{g,1}} \cong \HH_1(\Sigma_{g,1};\Q)$.

Trapp \cite{TrappLinear} investigated the action of $\Torelli_{g,1}$ on $\HH_1(U\Sigma_{g,1};\Q)$.  
In \cite[Proposition 2.8]{TrappLinear}, it is shown that there is a basis 
$\{z,\tilde{a}_1,\tilde{b}_1,\ldots,\tilde{a}_g,\tilde{b}_g\}$ for
$\HH_1(U\Sigma_{g,1};\Q)$ with the following three properties.  First, $z$ is the homology class of the fiber.  Second,
$\{\tilde{a}_1,\tilde{b}_1,\ldots,\tilde{a}_g,\tilde{b}_g\}$ projects to a symplectic basis for $\HH_1(\Sigma_{g,1};\Q)$.
Third, with respect to this basis, the image of $\Torelli_{g,1}$ in the automorphism group 
of $\HH_1(U\Sigma_{g,1};\Q) \cong \Q^{2g+1}$ consists of all matrices of the form
$\MatTwoTwo{1}{2v}{0}{\One}$.  Here $v$ is an arbitrary $2g$-dimensional row vector whose entries are integers and $\One$ is
the $2g \times 2g$-dimensional identity matrix.  In particular, some element of $\Torelli_{g,1}$ takes $\tilde{a}_1$
to $\tilde{a}_1 + 2z$.  We conclude that in $(\HH_1(U\Sigma_{g,1};\Q))_{\Torelli_{g,1}}$ these two elements
are equal, i.e.\ that $2z = 0$, as desired.
\end{proof}

\begin{lemma}
\label{lemma:h1modlh}
Let $L \geq 1$, $g \geq 3$, and $b,n \geq 0$.  Assume that $(b,n) \neq (0,0)$.
Then 
$$\HH_1(\Mod_{g,b}^n(L);\HH_1(\Sigma_{g,b}^n;\Q)) \cong \Q.$$
\end{lemma}
\begin{proof}
We have a short exact sequence
$$0 \longrightarrow \Q^{b+n-1} \longrightarrow \HH_1(\Sigma_{g,b}^n;\Q) \longrightarrow \HH_1(\Sigma_g;\Q) \longrightarrow 0$$
of $\Mod_{g,b}^n$-modules.  Here the action of $\Mod_{g,b}^n(L)$ on $\Q^{b+n-1}$ (generated by the loops around the boundary
components/punctures) is trivial.  Associated to this is a long exact sequences in $\Mod_{g,b}^n(L)$ homology.
Theorem \ref{theorem:hainh1} says that $\HH_1(\Mod_{g,b}^n(L);\Q^{b+n-1}) = 0$ and 
Lemma \ref{lemma:huhcoinv} says that 
$$\HH_0(\Mod_{g,b}^n(L);\HH_1(\Sigma_{g,b}^n;\Q)) \cong (\HH_1(\Sigma_{g,b}^n;\Q))_{\Mod_{g,b}^n(L)} = 0.$$
This long exact sequence thus contains the segment
$$0 \longrightarrow \HH_1(\Mod_{g,b}^n(L);\HH_1(\Sigma_{g,b}^n;\Q)) \longrightarrow \HH_1(\Mod_{g,b}^n(L);\HH_1(\Sigma_g;\Q)) \longrightarrow \Q^{b+n-1} \longrightarrow 0.$$
Theorem \ref{theorem:hainh1h} says that $\HH_1(\Mod_{g,b}^n(L);\HH_1(\Sigma_g;\Q)) \cong \Q^{b+n}$, and the lemma follows.
\end{proof}

\begin{lemma}
\label{lemma:h1modluh}
Let $L \geq 1$, $g \geq 3$, and $b,n \geq 0$.  Then 
$\HH_1(\Mod_{g,b}^n(L);\HH_1(U\Sigma_{g,b}^n;\Q)) = 0$.
\end{lemma}
\begin{proof}
If $b=n=0$, then Lemma \ref{lemma:h1uhclosed} says that
$\HH_1(U\Sigma_{g,b}^n;\Q) \cong \HH_1(\Sigma_{g,b}^n;\Q)$.
The lemma thus follows in this case from Theorem \ref{theorem:hainh1h}.  
Otherwise, $U\Sigma_{g,b}^n$ is a trivial $S^1$-bundle 
and we have a short exact sequence
$$1 \longrightarrow \Q \longrightarrow \HH_1(U\Sigma_{g,b}^n;\Q) \longrightarrow \HH_1(\Sigma_{g,b}^n;\Q) \longrightarrow 1$$
of $\Mod_{g,b}^n(L)$-modules.  Associated to this is a long exact sequence in $\Mod_{g,b}^n(L)$ homology.
Theorem \ref{theorem:hainh1} says that $\HH_1(\Mod_{g,b}^n(L);\Q) = 0$ and
Lemma \ref{lemma:huhcoinv} says that
$$\HH_0(\Mod_{g,b}^n(L);\HH_1(U\Sigma_{g,b}^n;\Q)) \cong (\HH_1(U\Sigma_{g,b}^n;\Q))_{\Mod_{g,b}^n(L)} = 0.$$
This long exact sequence thus contains the segment
$$0 \longrightarrow \HH_1(\Mod_{g,b}^n(L);\HH_1(U\Sigma_{g,b}^n;\Q)) \longrightarrow \HH_1(\Mod_{g,b}^n(L);\HH_1(\Sigma_{g,b}^n;\Q)) \longrightarrow \Q \longrightarrow 0.$$
Lemma \ref{lemma:h1modlh} says that $\HH_1(\Mod_{g,b}^n(L);\HH_1(\Sigma_{g,b}^n;\Q)) \cong \Q$, and the lemma follows.
\end{proof}

\subsection{Curve stabilizers}

We will also need the following generalization of Theorem \ref{theorem:hainh1}.

\begin{lemma}
\label{lemma:h1modlstab}
Fix $L \geq 2$ and $g,b,k \geq 0$ such that $g-k \geq 3$.  Let $\gamma_1, \ldots, \gamma_k$ be 
disjoint simple closed curves on $\Sigma_{g,b}$ such that
$\Sigma_{g,b} \setminus (\gamma_1 \cup \cdots \cup \gamma_k)$ is connected.  Then 
$\HH_1(\Mod_{g,b,\gamma_1,\ldots,\gamma_k}(L);\Q)=0$.
\end{lemma}

\begin{remark}
The notation $\Mod_{g,b,\gamma_1,\ldots,\gamma_k}(L)$ is defined above in Definition \ref{definition:curvestab}.
\end{remark}

\noindent
For the proof, we will need two definitions and two lemmas.

\Figure{figure:boundingpair}{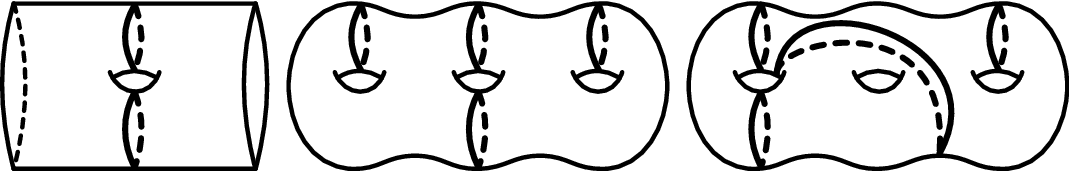}{a. The model for a standard bounding pair map
\CaptionSpace b. $T_{x_1} T_{x_2}^{-1}$ is a standard bounding pair 
map on $\Sigma_{3,\gamma_1,\gamma_2}$
\CaptionSpace c. $T_{x_1} T_{x_2}^{-1}$ is a generalized 
bounding pair map but not a standard bounding pair map on $\Sigma_{3,\gamma_1,\gamma_2}$}

\begin{definition}
A {\em separating twist} on $\Sigma_{g,b}^n$ is a Dehn twist $T_{\gamma}$, where $\gamma$ is a nontrivial separating
simple closed curve on $\Sigma_{g,b}^n$.  A {\em bounding pair map} on $\Sigma_{g,b}^n$ is a product
$T_{\gamma_1} T_{\gamma_2}^{-1}$, where $\gamma_1$ and $\gamma_2$ are disjoint nonisotopic nonseparating
simple closed curves on $\Sigma_{g,b}^n$ such that $\gamma_1 \cup \gamma_2$ separates $\Sigma_{g,b}^n$ (see
Figure \ref{figure:boundingpair}.a).
\end{definition}

\begin{remark}
Observe that separating twists and bounding pair maps lie in $\Torelli_{g,b}^n$.
Building on work of Birman \cite{BirmanSiegel}, Powell \cite{PowellTorelli} proved
that separating twists and bounding pair maps generate $\Torelli_{g,b}^n$ for $b=n=0$.  Later,
Johnson \cite{JohnsonFirst} showed that if $g \geq 3$, then only bounding pair maps are needed.  We will
slightly generalize this below in Lemma \ref{lemma:torelligenerators}.
\end{remark}

\begin{definition}
Fix $g,b,k \geq 0$ such that $g-k \geq 1$.  Let $\gamma_1,\ldots,\gamma_k$ be disjoint
simple closed curves on $\Sigma_{g,b}$ such that $\Sigma_{g,b} \setminus (\gamma_1 \cup \cdots \cup \gamma_k)$ is connected
and let $i : \Sigma_{g,b,\gamma_1,\ldots,\gamma_k} \hookrightarrow \Sigma_{g,b}$ be the inclusion.  Also,
let $T_{y_1} T_{y_2}^{-1} \in \Mod_{1,2}$ be the bounding pair map depicted in Figure \ref{figure:boundingpair}.a.
Then a {\em generalized bounding pair map} on $\Sigma_{g,b,\gamma_1,\ldots,\gamma_k}$ is
a mapping class $T_{x_1} T_{x_2}^{-1}$, where $x_1$ and $x_2$ are disjoint simple closed curves
on $\Sigma_{g,b,\gamma_1,\ldots,\gamma_k}$ such that $T_{i(x_1)} T_{i(x_2)}^{-1}$ is a
bounding pair map in $\Torelli_{g,b}$.  If, in addition, there exists an embedding 
$\Sigma_{1,2} \hookrightarrow \Sigma_{g,b,\gamma_1,\ldots,\gamma_k}$ that takes $y_1$ and $y_2$ to $x_1$ and $x_2$, respectively,
then we will say that $T_{x_1} T_{x_2}^{-1}$ is a {\em standard
bounding pair map}.  See Figure \ref{figure:boundingpair}.b--c.  Finally if $x$ is a simple closed
curve on $\Sigma_{g,b,\gamma_1,\ldots,\gamma_k}$ such that $T_{i(x)}$ is a separating twist on $\Sigma_{g,b}$, then we will say
that $T_x$ is a {\em generalized separating twist}.
\end{definition}

\begin{lemma}[{\cite[proof of Lemma 6.4]{PutmanPicardGroup}}]
\label{lemma:bpkill}
Fix $g,b,k \geq 0$ such that $g-k \geq 1$.  Let $T_{x_1} T_{x_2}^{-1}$ be a standard bounding pair map
on $\Sigma_{g,b,\gamma_1,\ldots,\gamma_k}$.  Then the class of $T_{x_1} T_{x_2}^{-1}$ in
$\HH_1(\Mod_{g,b,\gamma_1,\ldots,\gamma_k}(L);\Q)$ is trivial.
\end{lemma}

\begin{remark}
The proof of the main result of \cite{PutmanPicardGroup} depends on Theorem \ref{theorem:mainh2closed}, 
but the proof of \cite[Lemma 6.4]{PutmanPicardGroup} does not, so
no circularity is being introduced.  We decided to prove this lemma in \cite{PutmanPicardGroup}
because that paper required a slightly more precise result.
\end{remark}

\begin{remark}
If $k = 0$ and $b \leq 1$, then Lemma \ref{lemma:bpkill} is also contained in
the proof of a theorem of McCarthy \cite[Theorem 1.1]{McCarthyCofinite}.  One could also
adapt this proof to prove Lemma \ref{lemma:bpkill}.  
\end{remark}

\begin{lemma}
\label{lemma:torelligenerators}
Fix $g,b,k \geq 0$ such that $g-k \geq 3$.  Let $\gamma_1, \ldots, \gamma_k$ be disjoint 
simple closed curves on $\Sigma_{g,b}$ such that
$\Sigma_{g,b} \setminus (\gamma_1 \cup \cdots \cup \gamma_k)$ is connected.  Then
$\Torelli_{g,b,\gamma_1,\ldots,\gamma_k}$ is generated by its set of standard bounding pair maps.
\end{lemma}
\begin{proof}
By \cite[Theorem 1.3]{PutmanCutPaste},
the group $\Torelli_{g,b,\gamma_1,\ldots,\gamma_k}$ is generated by its set of generalized separating twists
and generalized bounding pair maps.  Our goal, therefore, 
is to express every generalized separating twist and generalized bounding pair 
map as a product of standard bounding pair maps.  Our main tool
will be the {\em lantern relation} (see \cite{PutmanInfinite}), which is the relation 
$$(T_{x_1}T_{x_2}^{-1})(T_{y_1}T_{y_2}^{-1})(T_{z_1}T_{z_2}^{-1}) = T_{\beta}$$
for curves $x_1, x_2, y_1, y_2, z_1, z_2$, and $\beta$ as shown in Figure \ref{figure:relations}.a.  

\Figure{figure:relations}{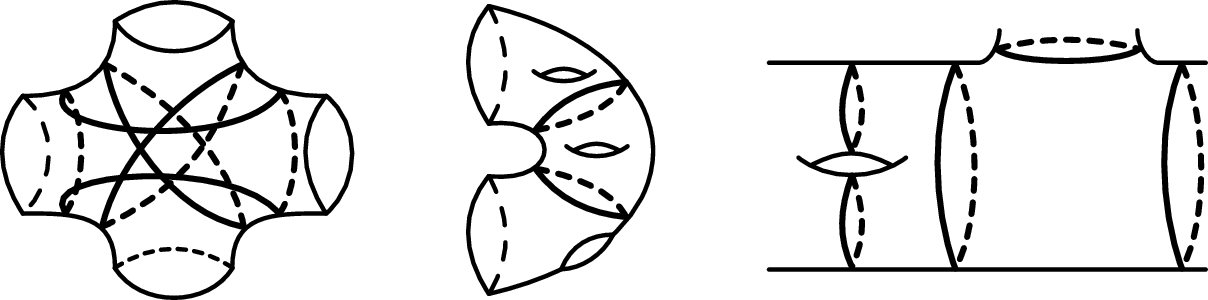}{a. The lantern relation
$(T_{x_1}T_{x_2}^{-1})(T_{y_1}T_{y_2}^{-1})(T_{z_1}T_{z_2}^{-1}) = T_{\beta}$ \CaptionSpace
b. A relation of the form $T_{a} T_{d}^{-1} = (T_a T_b^{-1})(T_b T_c^{-1})(T_c T_{d}^{-1})$ \CaptionSpace
c. Curves need for proof of Lemma \ref{lemma:torelligenerators}}

If $T_{\beta}$ is any generalized separating twist, then since $g-k \geq 3$ we can 
find a lantern relation
$$(T_{x_1}T_{x_2}^{-1})(T_{y_1}T_{y_2}^{-1})(T_{z_1}T_{z_2}^{-1}) = T_{\beta}$$
all of whose bounding pair maps are standard bounding pair maps.  This is the desired
expression.

Now let $T_{x_1} T_{x_2}^{-1}$ be any generalized bounding pair map.  It is easy to see that $T_{x_1} T_{x_2}^{-1}$
is a standard bounding pair map if and only if the $x_i$ are nonseparating curves on $\Sigma_{g,b,\gamma_1,\ldots,\gamma_k}$,
and moreover it is easy to see that is is impossible for one of the $x_i$ to be a separating curve and the other
to be a nonseparating curve.  Assume, therefore, that the $x_i$ are separating curves.  
Using a relation like that described in Figure \ref{figure:relations}.b,
we can write
$$T_{x_1} T_{x_2}^{-1} = (T_{y_1^1} T_{y_2^1}^{-1}) \cdots (T_{y_1^m} T_{y_2^m}^{-1}),$$
where for all $1 \leq i \leq m$, the mapping class $T_{y_1^i} T_{y_2^i}^{-1}$ is a generalized
bounding pair map such that the component $S_i$ of $\Sigma_{g,b,\gamma_1,\ldots,\gamma_k} \setminus (y_1^i \cup y_2^i)$
with $y_1^i,y_2^i \subset \partial S_i$ is homeomorphic to either $\Sigma_{1,2}$ or $\Sigma_{0,3}$.

Fix some $1 \leq i \leq m$.  Our goal is to write $T_{y_1^i} T_{y_2^i}^{-1}$ as a product
of standard bounding pair maps.  Flipping $y_1^i$ and $y_2^i$ (which has the harmless effect of
inverting $T_{y_1^i} T_{y_2^i}^{-1}$) if necessary, we can assume that
there is a positive genus component $T_i$ of $\Sigma_{g,b,\gamma_1,\ldots,\gamma_k} \setminus (y_1^i \cup y_2^i)$ with
$y_1^i \subset \partial T_i$ and $y_2^i \nsubseteq \partial T_i$.  We can then (see Figure \ref{figure:relations}.c) 
find a generalized separating
twist $T_{\beta_i}$ with $\beta_i \subset S_i$ and two simple closed curves $x_2^i,z_2^i \subset T_i$ that
do not separate $\Sigma_{g,b,\gamma_1,\ldots,\gamma_k}$ such that there
is a subsurface $U_i$ of $\Sigma_{g,b,\gamma_1,\ldots,\gamma_k}$ with $U_i \cong \Sigma_{0,4}$, with
$\partial U_i = x_2^i \cup y_2^i \cup z_2^i \cup \beta_i$, and with $y_1^i \subset U_i$.  There
are then simple closed curves $x_1^i$ and $z_1^i$ in $U_i$ such that we have a lantern
relation
$$T_{\beta_i} = (T_{x_1^i} T_{x_2^i}^{-1}) (T_{y_1^i} T_{y_1^i}^{-1}) (T_{z_1^i} T_{z_2^i}^{-1}),$$
where both $T_{x_1^i} T_{x_2^i}^{-1}$ and $T_{z_1^i} T_{z_2^i}^{-1}$ are generalized bounding pair maps.  Since
$x_2^i$ and $z_2^i$ are nonseparating curves, it follows that in fact $T_{x_1^i} T_{x_2^i}^{-1}$ and
$T_{z_1^i} T_{z_2^i}^{-1}$ are standard bounding pair maps.  Since we have already shown that
we can express the generalized separating twist $T_{\beta_i}$ as a product of standard bounding
pair maps, we are done.
\end{proof}

\begin{proof}[Proof of Lemma \ref{lemma:h1modlstab}]
Set $a_i = [\gamma_i]$ for $1 \leq i \leq k$ and let $\Sp_{2g}(\Z,L,a_1,\ldots,a_k)$ denote the image of 
$\Mod_{g,b,\gamma_1,\ldots,\gamma_k}(L)$ in $\Sp_{2g}(\Z,L)$.  We thus have an exact sequence
$$1 \longrightarrow \Torelli_{g,b,\gamma_1,\ldots,\gamma_k} \longrightarrow \Mod_{g,b,\gamma_1,\ldots,\gamma_k}(L) \longrightarrow
\Sp_{2g}(\Z,L,a_1,\ldots,a_k) \longrightarrow 1.$$
Combining Lemma \ref{lemma:torelligenerators} with Lemma \ref{lemma:bpkill}, we 
conclude that the image of 
$\HH_1(\Torelli_{g,b,\gamma_1,\ldots,\gamma_k};\Q)$ in $\HH_1(\Mod_{g,b,\gamma_1,\ldots,\gamma_k}(L);\Q)$ is zero.  The
five term exact sequence associated to our exact sequence thus degenerates into an isomorphism
$$\HH_1(\Mod_{g,b,\gamma_1,\ldots,\gamma_k}(L);\Q) \cong \HH_1(\Sp_{2g}(Z,L,a_1,\ldots,a_k);\Q).$$
Our goal is to prove that $\HH_1(\Sp_{2g}(Z,L,a_1,\ldots,a_k);\Q) = 0$.

The proof will be by induction on $k$.  The base case $k=0$ simply asserts that $\HH_1(\Sp_{2g}(\Z,L);\Q)=0$, which 
follows from the fact that $\Sp_{2g}(\Z,L)$ satisfies Kazhdan's property (T) (see, e.g.,\ \cite[Theorems 7.1.4 and
7.1.7]{ZimmerBook}).  
Assume now that $k>0$ and that the result is true for all smaller $k$.  Extend the 
$a_i$ to a symplectic basis $\{a_1,b_1,\ldots,a_g,b_g\}$ for $\HH_1(\Sigma_g;\Z)$, 
and consider $\phi \in \Sp_{2g}(\Z,L,a_1,\ldots,a_k)$.  Observe
that since $\phi(a_1) = a_1$, the map $\phi$ must preserve the subspace $V = \langle a_1,a_2,b_2,\ldots,a_g,b_g \rangle$.
We therefore get an induced map on the $2(g-1)$-dimensional symplectic $\Z$-module $V' = V / \langle a_1 \rangle$.  Letting
$a_i'$ be the image of $a_i$ in $V'$ for $2 \leq i \leq k$, we get a homomorphism
$$\pi : \Sp_{2g}(\Z,L,a_1,\ldots,a_k) \longrightarrow \Sp_{2(g-1)}(\Z,L,a_2',\ldots,a_k').$$
Moreover, there is clearly a right-inverse to $\pi$ (``extend by the identity'').  In other words,
for some group $K$ we have a split exact sequence
$$1 \longrightarrow K \longrightarrow \Sp_{2g}(\Z,L,a_1,\ldots,a_k) \stackrel{\pi}{\longrightarrow} \Sp_{2(g-1)}(\Z,L,a_2',\ldots,a_k') \longrightarrow 1.$$
This induces an isomorphism
\begin{align*}
\HH_1(\Sp_{2g}(\Z,L,a_1,\ldots,a_k);\Q) \cong &\HH_1(\Sp_{2(g-1)}(\Z,L,a_2',\ldots,a_k');\Q) \\
                                              &\oplus \HH_1(K;\Q)_{\Sp_{2(g-1)}(\Z,L,a_2',\ldots,a_k')}.
\end{align*}
By induction the first term is zero, so we must prove that the second term is zero.

Consider $\psi \in K$.  Letting $V'' = \langle a_1,\ldots,a_k,a_{k+1},b_{k+1},\ldots,a_g,b_g \rangle$, it is easy to see that
$\psi(b_1) = b_1 + v$ for some $v \in V''$.  We claim that
$\psi$ is determined by $v$.  Indeed, consider $s \in \{a_2,b_2,\ldots,a_g,b_g\}$.  Since $\psi \in K$,
there is some $k \in \Z$ such that $\psi(s) = s + k a_1$.  However, the integer $k$ is determined
by the fact that 
$$0 = i(b_1,s) = i(\psi(b_1),\psi(s)) = i(b_1 + v, s + k a_1),$$
whence the claim.  Conversely, any value of $v$ may occur, so the elements of $K$ are in bijection
with vectors in $V''$.  However, $K$ is not quite the additive group $V''$; indeed, setting
$V''' = \langle a_2, \ldots, a_k, a_{k+1}, b_{k+1}, \ldots, a_g, b_g \rangle$, it is easy to see that $K$ may be identified with
pairs 
$$\{\text{$(n,w)$ $|$ $n \in \Z$ and $w \in V''$}\},$$ 
where $n$ is the $a_1$-coordinate of the corresponding vector in $V'$ and
where we have the multiplication rule
$$(n_1, w_1) \cdot (n_2, w_2) = (n_1 + n_2 + i(w_1,w_2), w_1 + w_2).$$
The group $\Sp_{2(g-1)}(\Z,L,a_2',\ldots,a_k')$ acts on the second coordinate.
It is an easy exercise to see that with this description we have
$$\HH_1(K;\Q)_{\Sp_{2(g-1)}(\Z,L,a_2',\ldots,a_k')} = 0,$$
and we are done.
\end{proof}

\section{Reduction to the closed level $L$ subgroups}
\label{section:closedlevel}
The goal of this section is to reduce the proof of Theorem \ref{theorem:mainh2} to the following
theorem.

\begin{theorem}[{Rational $\HH_2$ of level $L$ subgroups, closed surfaces}]
\label{theorem:mainh2level}
For $L \geq 3$ and $g \geq 5$, we have $\HH_2(\Mod_{g}(L);\Q) \cong \HH_2(\Mod_{g};\Q)$.
\end{theorem}

\begin{remark}
While Theorem \ref{theorem:mainh2} applies $\Mod_{g,b}^n(L)$ for all $L \geq 2$, 
we emphasize that Theorem \ref{theorem:mainh2level} has the hypothesis $L \geq 3$.  This
stronger hypothesis is used below in the proof of Proposition \ref{proposition:decompthm}
to ensure that $\Mod_{g}(L)$ cannot reverse the orientation of a nonseparating simple closed curve.
\end{remark}

\subsection{Reduction to the level $L$ subgroups}

We first reduce the proof of Theorem \ref{theorem:mainh2} to the following theorem.

\begin{theorem}[{Rational $\HH_2$ of level $L$ subgroups, general surfaces}]
\label{theorem:mainh2levelbdry}
For $L \geq 3$, $g \geq 5$, and $b,n \geq 0$, we have $\HH_2(\Mod_{g,b}^n(L);\Q) \cong \HH_2(\Mod_{g,b}^n;\Q)$.
\end{theorem}

We will need the solution to the congruence subgroup problem for $\Sp_{2g}(\Z)$, which is
due to Mennicke \cite{MennickeCongruence} (see also \cite{BassMilnorSerre}).

\begin{theorem}[{Congruence subgroup problem for $\Sp_{2g}(\Z)$, \cite{MennickeCongruence}}]
\label{theorem:congruencesubgroupsp}
For $g \geq 2$, let $\Gamma$ be a finite index subgroup of $\Sp_{2g}(\Z)$.  Then there
is some $L \geq 2$ such that $\Sp_{2g}(\Z,L) < \Gamma$.
\end{theorem}

We will also need the following standard result, which follows easily from the Hochschild-Serre spectral
sequence and the existence of the so-called {\em transfer homomorphism} (c.f. \cite[Chapter III.9]{BrownCohomology}).

\begin{lemma}
\label{lemma:transfer}
If $G'$ is a finite-index subgroup of $G$ and $M$ is a $G$-vector space over $\Q$, then the map
$\HH_{\ast}(G';M) \rightarrow \HH_{\ast}(G;M)$ is surjective.  If $G'$ is a normal
subgroup of $G$, then $\HH_{\ast}(G;M) = (\HH_{\ast}(G';M))_G$, where the action of $G$ on $\HH_{\ast}(G';M)$
comes from the conjugation action of $G$ on $G'$.
\end{lemma}

\begin{proof}[{Proof of Theorem \ref{theorem:mainh2} assuming Theorem \ref{theorem:mainh2levelbdry}}]
The group $\Gamma$ fits into an exact sequence
$$1 \longrightarrow \Torelli_{g,b}^n \longrightarrow \Gamma \longrightarrow \Gamma' \longrightarrow 1,$$
where $\Gamma'$ is a finite index subgroup of $\Sp_{2g}(\Z)$.  By Theorem \ref{theorem:congruencesubgroupsp},
there is some $L \geq 2$ such that $\Sp_{2g}(\Z,L) < \Gamma'$.  Multiplying $L$ by $2$ if necessary, we
may assume that $L \geq 3$.  Pulling $\Sp_{2g}(\Z,L)$ back to $\Gamma$, we see
that $\Gamma$ contains $\Mod_{g,b}^n(L)$ as a subgroup of finite index.  
We have a sequence of maps
$$\HH_2(\Mod_{g,b}^n(L);\Q) \stackrel{i_1}{\longrightarrow} \HH_2(\Gamma;\Q) \stackrel{i_2}{\longrightarrow} \HH_2(\Mod_{g,b}^n;\Q).$$
Lemma \ref{lemma:transfer} says that $i_1$ and $i_2$ are surjective.  Also, Theorem \ref{theorem:mainh2levelbdry}
says that $\HH_2(\Mod_{g,b}^n(L);\Q) \cong \HH_2(\Mod_{g,b}^n;\Q)$, so we conclude that $i_2$ must be an
an isomorphism, as desired.
\end{proof}

\subsection{Eliminating the boundary components and punctures}

We now eliminate the boundary components and punctures.  
We will need the following result, which is an immediate consequence of the Gysin sequence.

\begin{lemma}
\label{lemma:eliminatepuncture}
Let $G$ be a group that fits into a central extension
$$1 \longrightarrow \Z \longrightarrow G \longrightarrow \Gamma \longrightarrow 1.$$
If $\HH_1(G;\Q) = 0$, then there is a natural short exact sequence
$$0 \longrightarrow \HH_2(G;\Q) \longrightarrow \HH_2(\Gamma;\Q) \longrightarrow \Q \longrightarrow 0.$$
\end{lemma}

\noindent
We will also need the following lemma.

\begin{lemma}
\label{lemma:killoffuh}
For $L \geq 2$, $g \geq 2$ and $b \geq 0$, we have $(\HH_2(\pi_1(U\Sigma_{g,b});\Q))_{\Mod_{g,b}(L)} = 0$.
\end{lemma}
\begin{proof}
By Lemma \ref{lemma:huhcoinv}, it is enough to show that we have an isomorphism
$$\HH_2(\pi_1(U\Sigma_{g,b});\Q) \cong \HH_1(\Sigma_{g,b};\Q)$$
of $\Mod_{g,b}(L)$-modules.  There are two cases.  In the first case, $b=0$, so 
$U\Sigma_{g,b}$ is a closed aspherical 3-manifold.  We then have
$$\HH_2(\pi_1(U\Sigma_{g,b});\Q) \cong \HH_2(U\Sigma_{g,b};\Q) \cong \HH^1(U\Sigma_{g,b};\Q) \cong \HH^1(\Sigma_{g,b};\Q) \cong \HH_1(\Sigma_{g,b};\Q).$$
The first isomorphism here follows from the fact that $U\Sigma_{g,b}$ is aspherical, the 
third isomorphism follows from Lemma \ref{lemma:h1uhclosed}, and the remaining isomorphisms are applications
of Poincar\'{e} duality.
In the second case, $b > 0$, so we have $U\Sigma_{g,b} \cong \Sigma_{g,b} \times S^1$, and the desired
result follows from the K\"{u}nneth formula and the fact that $\pi_1(\Sigma_{g,b})$ is free.
\end{proof}

\begin{proof}[{Proof of Theorem \ref{theorem:mainh2levelbdry} assuming Theorem \ref{theorem:mainh2level}}]
Denote the truth of Theorem \ref{theorem:mainh2levelbdry} for a surface $\Sigma_{g,b}^n$ by $F(g,b,n)$.  Hence
Theorem \ref{theorem:mainh2level} is asserting $F(g,0,0)$ for $g \geq 5$.  Also, Theorem
\ref{theorem:hainh1} says that we can apply Lemma \ref{lemma:eliminatepuncture} to
the central extension
$$1 \longrightarrow \Z \longrightarrow \Mod_{g,b+1}^n(L) \longrightarrow \Mod_{g,b}^{n+1}(L) \longrightarrow 1$$
that is induced by gluing a punctured disc to a boundary component of $\Sigma_{g,b+1}^n$.  This show that
$F(g,b+1,n)$ implies $F(g,b,n+1)$.  To prove the theorem, therefore, we must only prove that $F(g,b,0)$ implies
$F(g,b+1,0)$ for $g \geq 5$.

Consider the exact sequence
\begin{equation}
\label{eqn:eliminatebdryseq}
1 \longrightarrow \pi_1(U\Sigma_{g,b}) \longrightarrow \Mod_{g,b+1}(L) \longrightarrow \Mod_{g,b}(L) \longrightarrow 1
\end{equation}
given by Theorem \ref{corollary:birmanexactsequencel1}.  By assumption we have $\HH_2(\Mod_{g,b}(L);\Q) \cong \Q$.
Using Lemmas \ref{lemma:huhcoinv}, \ref{lemma:killoffuh} and \ref{lemma:h1modluh}, 
the $E^2$ page of the Hochschild-Serre spectral sequence associated
to \eqref{eqn:eliminatebdryseq} is of the form
\begin{center}
\begin{tabular}{|c@{\hspace{0.2 in}}c@{\hspace{0.2 in}}c}
$0$    &        &                           \\
$0   $ & $0$    &                           \\
$\ast$ & $\ast$ & $\HH_2(\Mod_{g,b}(L);\Q)$ \\
\cline{1-3}
\end{tabular}
\end{center}
The result follows.
\end{proof}

\section{The stability trick and the proof of the main theorem}
\label{section:stabilitytrick}
The key to our proof of Theorem \ref{theorem:mainh2level} will be the following lemma, which
says that if a finite-index normal subgroup of $\Mod_{g}$ satisfies a weak form of rational homological
stability (for a fixed homology group), then that homology group must be identical to that of $\Mod_{g}$.
If $\Gamma$ is a subgroup of $\Mod_g$ and $\gamma$ is a simple closed curve on $\Sigma_g$, then denote
by $\Gamma_{\gamma}$ the subgroup of $\Gamma$ stabilizing the isotopy class of $\gamma$.

\begin{lemma}[{Stability trick}]
\label{lemma:stabilitytrick}
For $g \geq 1$, let $\Gamma$ be a finite-index normal subgroup of $\Mod_{g}$.  Fix some integer $k$, and assume that for any
simple closed nonseparating curve $\gamma$, the map $\HH_k(\Gamma_{\gamma};\Q) \rightarrow \HH_k(\Gamma;\Q)$ is surjective.
Then $\HH_k(\Gamma;\Q) \cong \HH_k(\Mod_{g};\Q)$.
\end{lemma}
\begin{proof}
By Lemma \ref{lemma:transfer}, it is enough to show that the conjugation action
of $\Mod_{g}$ on $\Gamma$ induces the trivial action on $\HH_k(\Gamma;\Q)$.  To do this, it is sufficient to
check that a Dehn twist $T_{\gamma}$ about a nonseparating curve $\gamma$ on $\Sigma_{g}$
acts trivially on $\HH_k(\Gamma;\Q)$.  Since $T_{\gamma}$ is central in $\Gamma_{\gamma}$, it acts trivially
on $\HH_k(\Gamma_{\gamma};\Q)$, so by assumption it also acts trivially on $\HH_k(\Gamma;\Q)$, and we are done.
\end{proof}

We now prove Theorem \ref{theorem:mainh2level}, making use of two results whose proofs are postponed.
Recall that $\CNosep_{g}$ is the simplicial complex whose $(n-1)$-simplices are sets $\{\gamma_1,\ldots,\gamma_n\}$ of isotopy
class of nonseparating simple closed curves that can be realized disjointly with 
$\Sigma_{g} \setminus (\gamma_1 \cup \cdots \cup \gamma_n)$ connected.

\begin{proof}[{Proof of Theorem \ref{theorem:mainh2level}}]
Fix some $g \geq 5$ and $L \geq 3$.  We must check the conditions of Lemma \ref{lemma:stabilitytrick} for 
$\Gamma = \Mod_{g}(L)$ and $k=2$.
Recalling the notation $\Mod_{g,\gamma}(L)$ from Definition \ref{definition:curvestab}, 
for any simple closed nonseparating curve $\gamma$ we have a factorization
$$\HH_2(\Mod_{g,\gamma}(L);\Q) \longrightarrow \HH_2((\Mod_g(L))_{\gamma};\Q) \longrightarrow \HH_2(\Mod_g(L);\Q).$$
It is thus enough to prove that the map $\HH_2(\Mod_{g,\gamma}(L);\Q) \rightarrow \HH_2(\Mod_g(L);\Q)$ is surjective.
The first step is Proposition \ref{proposition:decompthm} below, which says that
$\HH_2(\Mod_{g}(L);\Q)$ is ``carried'' on curve stabilizers.  This proposition
will be proven in \S \ref{section:decompthm}.

\begin{proposition}[{Decomposition theorem}]
\label{proposition:decompthm}
For $g \geq 5$ and $L \geq 3$ the natural map
$$\bigoplus_{\gamma \in (\CNosep_{g})^{(0)}} \HH_2(\Mod_{g,\gamma}(L);\Q) \longrightarrow \HH_2(\Mod_{g}(L);\Q)$$
is surjective.
\end{proposition}

The second step is Proposition \ref{proposition:weakstability} below, which will be proven
in \S \ref{section:weakstability}.

\Figure{figure:weakstability}{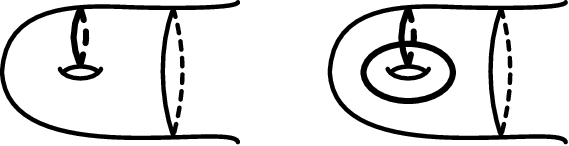}{a. $\Sigma_{g} \setminus S$ is a one-holed torus containing $\gamma$. \CaptionSpace
b. The same surface $S$ works for both $\gamma_1$ and $\gamma_2$}

\begin{proposition}[{Weak stability theorem}]
\label{proposition:weakstability}
Fix $g \geq 5$ and $L \geq 2$.  Let $\gamma$ be a nonseparating simple closed curve on $\Sigma_{g}$. 
Also, let $S$ be a subsurface
of $\Sigma_{g}$ such that $S \cong \Sigma_{g-1,1}$ and such that $S$ is embedded in $\Sigma_{g}$
as depicted in Figure \ref{figure:weakstability}.a (in particular, $\gamma \subset \Sigma_{g} \setminus S$).
Then the natural map $\HH_2(\Mod(S,L);\Q) \rightarrow \HH_2(\Mod_{g,\gamma}(L);\Q)$ is an
isomorphism.
\end{proposition}

If $\gamma_1$ and $\gamma_2$ are two nonseparating simple closed curves that intersect once, then
we can find a subsurface $S$ of $\Sigma_{g}$ such that $S \cong \Sigma_{g-1,1}$ and such that
$S$ is embedded in $\Sigma_{g}$ as depicted in Figure \ref{figure:weakstability}.b.  We have
a commutative diagram
$$\begin{CD}
\HH_2(\Mod(S,L);\Q)            @>>> \HH_2(\Mod_{g,\gamma_1}(L);\Q)\\
@VVV                                @VVV\\
\HH_2(\Mod_{g,\gamma_2}(L);\Q) @>>> \HH_2(\Mod_g(L);\Q)
\end{CD}$$
Proposition \ref{proposition:weakstability} says that the natural
map $\HH_2(\Mod(S,L);\Q) \rightarrow \HH_2(\Mod_{g,\gamma_i}(L);\Q)$ is an isomorphism for $1 \leq i \leq 2$.  This
implies that the images of $\HH_2(\Mod_{g,\gamma_1}(L);\Q)$ and
$\HH_2(\Mod_{g,\gamma_2}(L);\Q)$ in $\HH_2(\Mod_{g}(L);\Q)$ are equal.  

It is well-known (see, e.g., \cite[Lemma A.2]{PutmanCutPaste}) 
that for any two nonseparating simple closed curves $\gamma$ and $\gamma'$ on
$\Sigma_{g}$, there is a sequence
$$\gamma = \alpha_1, \alpha_2, \ldots, \alpha_k = \gamma'$$
of nonseparating simple closed curves on $\Sigma_{g}$ such that $\alpha_i$ and $\alpha_{i+1}$ intersect
once for $1 \leq i < k$.  We conclude that each factor of
$$\bigoplus_{\gamma \in (\CNosep_{g})^{(0)}} \HH_2(\Mod_{g,\gamma}(L);\Q)$$
has the same image in $\HH_2(\Mod_{g}(L);\Q)$.  Proposition \ref{proposition:decompthm}
thus implies that for any nonseparating simple closed curve $\gamma$, the map
$\HH_2(\Mod_{g,\gamma}(L);\Q) \longrightarrow \HH_2(\Mod_{g}(L);\Q)$ is surjective, and
the theorem follows.
\end{proof}

\section{Proof of the decomposition theorem}
\label{section:decompthm}
This section has four parts.  First, in \S \ref{section:decompthmweak} we prove a slightly weakened
version of Proposition \ref{proposition:decompthm}, making use of a certain connectivity
result whose proof we postpone.  Next, in \S \ref{section:decompthmstrong} we strengthen our result
to prove Proposition \ref{proposition:decompthm}.  Finally, in \S \ref{section:cnosepmodlcon} we prove the
aforementioned connectivity result.

\subsection{The nonseparating complex of curves and a weak version of Proposition \ref{proposition:decompthm}}
\label{section:decompthmweak}

In this section, we prove a slight weakening of Proposition \ref{proposition:decompthm}.  Our
main tool will be a certain theorem arising from the theory of equivariant homology that gives
a decomposition of the homology groups of a group acting on a simplicial complex.  This result
is usually stated in terms of a certain spectral sequence, but in our situation only one relevant term
of the spectral sequence is non-zero, so we are able to avoid even mentioning it.  First, a
definition.

\begin{definition}
A group $G$ acts on a simplicial complex $X$ {\em without rotations} if for all
simplices $s$ of $X$, the stabilizer $G_s$ fixes $s$ pointwise.
\end{definition}

The result we need 
is the following.  It follows easily from the two spectral sequences given in \cite[Chapter VII.7]{BrownCohomology}.

\begin{theorem}
\label{theorem:equivarianthomologythm}
Let $R$ be a ring, and consider a group $G$ acting without rotations on a simply connected simplicial complex $X$.  Assume
that $X/G$ is $2$-connected and that for any $\{v,v'\} \in X^{(1)}$ we have $\HH_1(G_{\{v,v'\}};R) = 0$.  Then the natural map
$$\bigoplus_{v \in X^{(0)}} \HH_2(G_v;R) \longrightarrow \HH_2(G;R)$$
is surjective.
\end{theorem}

To apply this to our situation, will need the following theorem of Harer.

\begin{theorem}[{Harer, \cite[Theorem 1.1]{HarerStability}}]
\label{theorem:cnosepcon}
$\CNosep_{g}$ is $(g-2)$-connected.
\end{theorem}

In \S \ref{section:cnosepmodlcon}, we will prove the following result, which is a
variant of \cite[Proposition 4.4]{PutmanInfinite}.

\begin{proposition}
\label{proposition:cnosepmodlcon}
For $L \geq 3$ and $g \geq 2$, the space $\CNosep_{g} / \Mod_{g}(L)$ is $(g-2)$-connected.
\end{proposition}

First, however, we will prove Proposition \ref{proposition:decompthmweak} below. 
The statement of it resembles Proposition \ref{proposition:decompthm},
but instead of the groups $\Mod_{g,\gamma}(L)$ on the ``cut'' surface used in Proposition
\ref{proposition:decompthm}, it uses the stabilizer subgroups $(\Mod_{g}(L))_{\gamma}$.  As
we will see in \S \ref{section:decompthmstrong}, the groups $\HH_2((\Mod_{g}(L))_{\gamma};\Q)$ are
slightly bigger than the groups $\HH_2(\Mod_{g,\gamma}(L);\Q)$, so this is a slight weakening
of Proposition \ref{proposition:decompthm}.

\begin{proposition}
\label{proposition:decompthmweak}
For $g \geq 5$ and $L \geq 3$, the natural map
$$\bigoplus_{\gamma \in (\CNosep_{g})^{(0)}} \HH_2((\Mod_{g}(L))_{\gamma};\Q) \longrightarrow \HH_2(\Mod_{g}(L);\Q)$$
is surjective.
\end{proposition}
\begin{proof}
Since the curves in a simplex of $\CNosep_{g}$ all define different classes in $\HH_1(\Sigma_g;\Z/L)$, the
group $\Mod_{g}(L)$ acts without rotations on $\CNosep_{g}$.  Theorem \ref{theorem:cnosepcon} and 
Lemma \ref{lemma:h1modlstab} together with Proposition \ref{proposition:cnosepmodlcon}
thus imply that the action of $\Mod_{g}(L)$ on $\CNosep_g$ satisfies the conditions of Theorem
\ref{theorem:equivarianthomologythm}.  The proposition follows.
\end{proof}

\subsection{The proof of Proposition \ref{proposition:decompthm}}
\label{section:decompthmstrong}

\begin{proof}[{Proof of Proposition \ref{proposition:decompthm}}]
Consider a simple closed nonseparating curve $\delta$ on $\Sigma_g$.  Since $L \geq 3$, the stabilizer
subgroup $(\Mod_{g}(L))_{\delta}$ cannot reverse the orientation of $\delta$.  
Letting $\delta_1$ and $\delta_2$ be the boundary curves of $\Sigma_{g,\delta}$, we thus have a central extension
\begin{equation}
\label{eqn:centralextension}
1 \longrightarrow \Z \longrightarrow \Mod_{g,\delta}(L) \longrightarrow (\Mod_{g}(L))_{\delta} \longrightarrow 1,
\end{equation}
where the kernel $\Z$ is generated by $T_{\delta_1} T_{\delta_2}^{-1}$.  Using Lemma
\ref{lemma:h1modlstab}, we can apply Lemma \ref{lemma:eliminatepuncture}
and conclude that we have a short exact sequence
$$0 \longrightarrow \HH_2(\Mod_{g,\delta}(L);\Q) \longrightarrow \HH_2((\Mod_{g}(L))_{\delta};\Q) \longrightarrow \Q
\longrightarrow 0.$$
To deduce Proposition \ref{proposition:decompthm} from Proposition \ref{proposition:decompthmweak},
we must show that the image in $\HH_2(\Mod_g(L);\Q)$ of one of the other summands of
$$\bigoplus_{\gamma \in (\CNosep_{g})^{(0)}} \HH_2(\Mod_{g,\gamma}(L);\Q)$$
contains the image of a complement to $\HH_2(\Mod_{g,\delta}(L);\Q)$ in $\HH_2((\Mod_{g}(L))_{\delta};\Q)$.

Choose an embedded subsurface $\Sigma_{4,1} \hookrightarrow \Sigma_{g}$ with $\delta \subset \Sigma_{4,1}$.  We can
expand \eqref{eqn:centralextension} to a commutative diagram of central extensions
$$\begin{CD}
1 @>>> \Z @>>> \Mod_{4,1,\delta}(L) @>>> (\Mod_{4,1}(L))_{\delta} @>>> 1\\
@.     @VV{=}V @VVV                   @VVV                     @.\\
1 @>>> \Z @>>> \Mod_{g,\delta}(L)   @>>> (\Mod_{g}(L))_{\delta} @>>> 1
\end{CD}$$
By Lemmas \ref{lemma:h1modlstab} and Lemma \ref{lemma:eliminatepuncture}, we have a corresponding commutative diagram
of short exact sequences
$$\begin{CD}
0 @>>> \HH_2(\Mod_{4,1,\delta}(L);\Q) @>>> \HH_2((\Mod_{4,1}(L))_{\delta};\Q) @>>> \Q @>>> 0\\
@.     @VVV                                @VVV                                    @VV{\cong}V @.\\
0 @>>> \HH_2(\Mod_{g,\delta}(L);\Q) @>>> \HH_2((\Mod_{g}(L))_{\delta};\Q) @>>> \Q @>>> 0
\end{CD}$$
This implies that the image of $\HH_2((\Mod_{4,1}(L))_{\delta};\Q)$ in $\HH_2((\Mod_{g}(L))_{\delta};\Q)$
contains a complement to $\HH_2(\Mod_{g,\delta}(L);\Q)$ in $\HH_2((\Mod_{g}(L))_{\delta};\Q)$.
Let $\delta'$ be any nonseparating simple closed curve on $\Sigma_g$ that is disjoint from $\Sigma_{4,1}$.  The
key observation is that we have a commutative diagram 
$$\begin{CD}
(\Mod_{4,1}(L))_{\delta}    @>>> (\Mod_{g}(L))_{\delta} \\
@VVV                             @VVV \\
\Mod_{g,\delta'}(L)        @>>> \Mod_g(L)
\end{CD}$$
Thus the image of $\HH_2(\Mod_{g,\delta'}(L);\Q)$ in $\HH_2(\Mod_g(L);\Q)$ contains
the image in $\HH_2(\Mod_g(L);\Q)$ 
of $\Image(\HH_2((\Mod_{4,1}(L))_{\delta};\Q) \rightarrow \HH_2((\Mod_{g}(L))_{\delta};\Q))$, and we
are done.
\end{proof}

\subsection{The proof of Proposition \ref{proposition:cnosepmodlcon}}
\label{section:cnosepmodlcon}

In \S \ref{section:linalg} we give a linear-algebraic
reformulation of Proposition \ref{proposition:cnosepmodlcon}.  The actual proof
is in \S \ref{section:laxisobases}.

\subsubsection{A linear-algebraic reformulation of Proposition \ref{proposition:cnosepmodlcon}}
\label{section:linalg}

We will need the following definition.

\begin{definition}
Fix $L \geq 0$ and $g \geq 1$.
\begin{itemize}
\item A {\em primitive vector} $v \in \HH_1(\Sigma_g;\Z/L)$ is a nonzero vector such that if $w \in \HH_1(\Sigma_g;\Z/L)$
satisfies $v = c \cdot w$ for some $c \in \Z/L$, then $c$ is a unit.
\item A {\em lax primitive vector} in $\HH_1(\Sigma_g;\Z/L)$ is a pair $\{v,-v\}$, where $v$ is a primitive vector.  
We will denote this pair by $\pm v$.
\item A {\em $k$-dimensional lax isotropic basis} in $\HH_1(\Sigma_g;\Z/L)$ 
is a set $\{\pm v_1, \ldots \pm v_k\}$ of lax primitive vectors such that $i(v_i,v_j) = 0$ for all $1 \leq i,j \leq k$
and such that $\Span{v_1,\ldots,v_k}$ is a summand of $V$ that is isomorphic to a $k$-dimensional free $\Z/L$-module.
\item We will denote by $\Bases(g,L)$ the simplicial complex whose $(k-1)$-simplices are $k$-dimensional lax isotropic
bases in $\HH_1(\Sigma_g;\Z/L)$.  Also, if $\Delta$ is either $\emptyset$ or a simplex of $\Bases(g,L)$, then
we will denote by $\Bases^{\Delta}(g,L)$ the simplicial complex consisting of the link of $\Delta$ in $\Bases(g,L)$.
Finally, if $W$ is an arbitrary $\Z/L$-submodule of $\HH_1(\Sigma_g;\Z/L)$, 
then we will denote by $\Bases^{\Delta,W}(g,L)$ the subcomplex
of $\Bases^{\Delta}(g,L)$ consisting of simplices $\{\pm v_1,\ldots, \pm v_k\}$ such that
$v_i \in W$ for $1 \leq i \leq k$.
\end{itemize}
\end{definition}

We then have the following result.

\begin{proposition}
For $g \geq 1$ and $L \geq 2$ we have $\CNosep_g / \Mod_g(L) \cong \Bases(g,L)$.
\end{proposition}
\begin{proof}
By \cite[Lemma 6.2]{PutmanCutPaste}, we have
$$\CNosep_g / \Torelli_g \cong \Bases(g,0).$$
The space $\CNosep_g / \Mod_{g}(L)$ is the quotient of $\Bases(g,0)$ by $\Mod_{g}(L) / \Torelli_g \cong \Sp_{2g}(\Z,L)$,
and we have a surjection $\pi : \Bases(g,0) \rightarrow \Bases(g,L)$ that is invariant under
the action of $\Sp_{2g}(\Z,L)$.  Moreover, two $(k-1)$-simplices $s=\{\pm v_1,\ldots,\pm v_k\}$ and
$s'=\{\pm w_1,\ldots,\pm w_k\}$ of $\Bases(g,0)$ are in the same $\Sp_{2g}(\Z,L)$-orbit if and only if
after possibly reordering the $\pm v_i$ we have that $\pm v_i$ and $\pm w_i$ map to the
same lax vector in $\HH_1(\Sigma_g;\Z/L)$ for all $1 \leq i \leq k$, i.e.\ if and only if $\pi(s)=\pi(s')$.  Finally, since
no two lax vectors in a simplex of $\Bases(g,0)$ can map to the same lax vector in $\HH_1(\Sigma_g;\Z/L)$,
it follows that $\Sp_{2g}(\Z,L)$ acts without rotations on $\Bases(g,0)$.  The proposition follows.
\end{proof}

We conclude that Proposition \ref{proposition:cnosepmodlcon} is equivalent to a
special case of the following proposition, whose proof is in \S \ref{section:laxisobases}.  This
proposition is related to a theorem of Charney \cite[Theorem 2.9]{CharneyVogtmann}, and the proof
is a variant of the proof of \cite[Proposition 6.14, conclusion 2]{PutmanInfinite}.

\begin{proposition}
\label{proposition:basesconnected}
Fix $g \geq 1$ and $L \geq 2$ and $0 \leq k \leq g$.  
Let $\{a_1,b_1,\ldots,a_g,b_g\}$ be a symplectic basis for $\HH_1(\Sigma_g;\Z/L)$.
Set $W = \Span{a_1,b_1,\ldots,a_{g-1},b_{g-1},a_g}$.  Also, set
$\Delta^k = \{\Span{a_1},\ldots,\Span{a_k}\}$ if $k \geq 1$ and $\Delta^k = \emptyset$
if $k = 0$.  Then for $-1 \leq n \leq g-k-2$, we have $\pi_n(\Bases^{\Delta^k}(g,L))=0$ and $\pi_n(\Bases^{\Delta^k,W}(g,L))=0$.
\end{proposition}

\subsubsection{The complex of lax isotropic bases}
\label{section:laxisobases}

We will need the following definition.

\begin{definition}
Assume that a symplectic basis $\{a_1,b_1,\ldots,a_g,b_g\}$ for $\HH_1(\Sigma_g;\Z/L)$ has been fixed and that
$\rho \in \{a_1,b_1,\ldots,a_g,b_g\}$.  Consider a lax primitive vector $\pm v$ in $\HH_1(\Sigma_g;\Z/L)$.
Express $v$ as $\sum (c_{a_i} a_i + c_{b_i} b_i)$ with $c_{a_i},c_{b_i} \in \Z/L$
for $1 \leq i \leq g$.  Letting $|x|$ for $x \in \Z/L$ be the unique integer representing
$x$ with $0 \leq |x| < L$, we define the {\em $\rho$-rank} of $\pm v$ to equal 
$\min\{|c_{\rho}|,|-c_{\rho}|\}$.  We will denote the $\rho$-rank of $\pm v$ by $\Rank^{\rho}(\pm v)$.
\end{definition}

\begin{proof}[{Proof of Proposition \ref{proposition:basesconnected}}]
Let $C^{\Delta^k}$ be $\Bases^{\Delta^k,W}(g,L)$ or $\Bases^{\Delta^k}(g,L)$.  We must prove
that $\pi_n(C^{\Delta^k}) = 0$ for $-1 \leq n \leq g-k-2$.  In the course of our proof, we will use the
case of $C^{\Delta^k} = \Bases^{\Delta^k,W}(g,L)$ to deal with the case of $C^{\Delta^k} = \Bases^{\Delta^k}(g,L)$; the
reader will easily verify that no circular reasoning is involved.  The proof will be by
induction on $n$.  The base case $n=-1$ is equivalent to the observation that if $k < g$, then
both $\Bases^{\Delta^k,W}(g,L)$ and $\Bases^{\Delta^k}(g,L)$ are nonempty.
Assume now that $0 \leq n \leq g-k-2$ and that
$\pi_{n'}(\Bases^{\Delta^{k'},W}(g,L)) = \pi_{n'}(\Bases^{\Delta^{k'}}(g,L))=0$ for all
$0 \leq k' < g$ and $-1 \leq n' \leq g-k'-2$ such that $n' < n$.  Let $S$ be a combinatorial $n$-sphere
and let $\phi : S \rightarrow C^{\Delta^k}$ be a simplicial map.  By Lemma \ref{lemma:simpapprox},
it is enough to show that $\phi$ may be homotoped to a constant map.
Let $\rho$ equal $a_g$ if $C^{\Delta^k} = \Bases^{\Delta^k,W}(g,L)$ and $b_g$ if $C^{\Delta^k} = \Bases^{\Delta^k}(g,L)$.  Set
$$R = \Max \{\text{$\Rank^{\rho}(\phi(x))$ $|$ $x \in S^{(0)}$}\}.$$
If $R = 0$ and $C^{\Delta^k} = \Bases^{\Delta^k,W}(g,L)$, then $\phi(S) \subset \Star_{\Bases^{\Delta^k,W}(g,L)}(\pm a_g)$,
and hence the map $\phi$ can be homotoped to the constant map $\pm a_g$.  If
$R = 0$ and $C^{\Delta^k} = \Bases^{\Delta^k}(g,L)$, then $\phi(S) \subset \Bases^{\Delta^k,W}(g,L)$, and
hence by the $C^{\Delta^k} = \Bases^{\Delta^k,W}(g,L)$ case we can homotope $\phi$ to a constant map.

Assume, therefore, that $R>0$.  Let $\Delta'$ be a simplex of $S$ such that 
$\Rank^{\rho}(\phi(x))=R$ for all vertices $x$ of $\Delta'$.  Choose
$\Delta'$ so that $m := \Dim(\Delta')$ is maximal, which implies that $\Rank^{\rho}(\phi(x)) < R$
for all vertices $x$ of $\Link_S(\Delta')$.  Now, $\Link_S(\Delta')$ is a combinatorial $(n-m-1)$-sphere and
$\phi(\Link_S(\Delta'))$ is contained in
$$\Link_{C^{\Delta^k}}(\phi(\Delta')) \cong C^{\Delta^{k+m'}}$$
for some $m' \leq m$ (it may be less than $m$ if $\phi|_{\Delta'}$ is not injective).  The inductive hypothesis together
with Lemma \ref{lemma:simpapprox} therefore tells us that there a combinatorial $(n-m)$-ball $B$ with
$\partial B = \Link_S(\Delta')$ and a simplicial map $f : B \rightarrow \Link_{C^{\Delta^k}}(\phi(\Delta'))$ such that
$f|_{\partial B} = \phi|_{\Link_S(\Delta')}$.

Our goal now is to adjust $f$ so that $\Rank^{\rho}(\phi(x)) < R$ for all $x \in B^{(0)}$.
Let $v \in \HH_1(\Sigma_g;\Z/L)$ be a vector whose $\rho$-coordinate equals $R$ modulo $L$ such that 
$\pm v$ is a vertex in $\phi(\Delta')$.
We define a map $f' : B \rightarrow \Link_{C^{\Delta^k}}(\phi(\Delta'))$ in the following way.  Consider
$x \in B^{(0)}$.  Let $v_x \in \HH_1(\Sigma_g;\Z/L)$ be a vector with $f(x) = \pm v_x$ whose $\rho$-coordinate equals
$\Rank^{\rho}(f(x))$ modulo $L$.  
By the division algorithm, there exists some $q_x \in \Z/L$
such that $\Rank^{\rho}(\pm(v_x + q_x v)) < R$.  Moreover, by the maximality of $m$ we can choose $q_x$
such that $q_x = 0$ if $x \in (\partial B)^{(0)}$.  
Define $f'(x) = \pm(v_x + q_x v)$.  It is clear that the map
$f'$ extends to a map $f' : B \rightarrow \Link_{C^{\Delta^k}}(\phi(\Delta'))$.  Additionally,
$f'|_{\partial B} = f|_{\partial B} = \phi|_{\Link_S(\Delta')}$.
We conclude that we can homotope $\phi$ so as to replace $\phi|_{\Star_S(\Delta')}$ with $f'$.  Since
$\Rank^{\rho}(f'(x)) < R$ for all $x \in B$, we have removed $\Delta'$ from $S$ without introducing any
vertices whose images have $\rho$-rank greater than or equal to $R$.  Continuing in this manner allows
us to simplify $\phi$ until $R=0$, and we are done.
\end{proof}

\section{Proof of the weak stability theorem}
\label{section:weakstability}

In this section, we prove Proposition \ref{proposition:weakstability} (the weak stability theorem).
Our main tool 
will be Theorem \ref{theorem:birmanexactsequencel2}, which we recall gives a split exact sequence
\begin{equation}
\label{eqn:levelexactsequence}
1 \longrightarrow \overline{K}_{g-1,1} \longrightarrow \Mod_{g,\gamma}(L) \longrightarrow \Mod_{g-1,1}(L) \longrightarrow 1.
\end{equation}
Here $\overline{K}_{g-1,1}$ fits into an exact sequence
$$1 \longrightarrow \Z \longrightarrow \overline{K}_{g-1,1} \longrightarrow K_{g-1,1} \longrightarrow 1,$$
where
$$K_{g-1,1} \cong \Ker(\pi_1(\Sigma_{g-1,1}) \longrightarrow \HH_1(\Sigma_{g-1,1},\Z/L))$$
and where $\Z$ is generated by $T_{\beta}^L$ with $\beta$ one of the boundary curves of $\Sigma_{g,\gamma}$.

We will need the following lemma, which is an easy consequence of the Hochschild-Serre spectral sequence.

\begin{lemma}
\label{lemma:h2decomplemma}
If
$$1 \longrightarrow A \longrightarrow B \longrightarrow C \longrightarrow 1$$
is a split exact sequence of groups, then there is an un-natural isomorphism
$\HH_2(B;\Q) \cong \HH_2(C;\Q) \oplus \HH_1(C;\HH_1(A;\Q)) \oplus D$,
where $D \cong \Image(\HH_2(A;\Q) \rightarrow \HH_2(B;\Q))$.
\end{lemma}

Let $C_{g,1}$ be the kernel of the natural map $\HH_1(K_{g,1};\Q) \rightarrow \HH_1(\Sigma_{g,1};\Q)$.  In
\cite[\S 5.1]{PutmanVanishing}, the author proved two things.  First, there is a $\Mod_{g,1}$-equivariant splitting
$\HH_1(K_{g,1};\Q) \cong \HH_1(\Sigma_{g,1};\Q) \oplus C_{g,1}$.
Second, we have $\HH_1(\Mod_{g,1}(L);C_{g,1}) = 0$ for $g \geq 4$ and
$L \geq 2$.  We thus obtain the following result.

\begin{lemma}
\label{lemma:vanishing}
For $g \geq 4$ and $L \geq 2$, the natural map 
$$\HH_1(\Mod_{g,1}(L);\HH_1(K_{g,1};\Q)) \longrightarrow \HH_1(\Mod_{g,1}(L);\HH_1(\Sigma_{g,1};\Q))$$
is an isomorphism.
\end{lemma}

We now prove the following lemma.

\begin{lemma}
\label{lemma:weakstabilityvanish}
If $g \geq 4$ and $L \geq 2$, then $\HH_1(\Mod_{g,1}(L);\HH_1(\overline{K}_{g,1};\Q)) = 0$.
\end{lemma}
\begin{proof}
We have a commutative diagram of central extensions
$$\begin{CD}
1 @>>> \Z   @>>> \overline{K}_{g,1}               @>>> K_{g,1}                       @>>> 1  \\
@.     @VVV      @VVV                              @VVV                      @. \\
1 @>>> \Z   @>>> \pi_1(U\Sigma_{g,1})   @>>> \pi_1(\Sigma_{g,1}) @>>> 1
\end{CD}$$
Here the left hand vertical map $\Z \rightarrow \Z$ is multiplication by $L$ and all the vertical maps
are injections.  This induces a map between the associated 5-term exact sequences,
but since $\pi_1(\Sigma_{g,1})$ and $K_{g,1}$ are free
this degenerates into a commutative diagram
\begin{equation}
\label{eqn:h1kh1usdiagram}
\begin{CD}
0 @>>> \Q   @>>> \HH_1(\overline{K}_{g,1};\Q)             @>>> \HH_1(K_{g,1};\Q)                       @>>> 0  \\
@.     @VV{\cong}V      @VVV                                      @VVV                                     @. \\
0 @>>> \Q   @>>> \HH_1(U\Sigma_{g,1};\Q) @>>> \HH_1(\Sigma_{g,1};\Q) @>>> 0
\end{CD}
\end{equation}
Here the top row is a short exact sequence of $\Mod_{g,1}(L)$-modules, the bottom row is a short
exact sequence of $\Mod_{g,1}$-modules, and the vertical arrows are equivariant with respect to the
inclusion map $\Mod_{g,1}(L) \hookrightarrow \Mod_{g,1}$.  The top (resp. bottom) short exact sequence
in \eqref{eqn:h1kh1usdiagram} induces a long exact sequence in $\Mod_{g,1}(L)$ (resp. $\Mod_{g,1}$) homology, and we get
an induced map between these two long exact sequences.

Theorem \ref{theorem:hainh1} says that $\HH_1(\Mod_{g,1}(L);\Q) = 0$.  Lemma \ref{lemma:h1modluh} 
says that $\HH_1(\Mod_{g,1};\HH_1(U\Sigma_{g,1};\Q)) = 0$.
A portion of the map between long exact sequences arising from \eqref{eqn:h1kh1usdiagram} thus looks like the following.
$$\minCDarrowwidth19pt\begin{CD}
0 @>>> \HH_1(\Mod_{g,1}(L);\HH_1(\overline{K}_{g,1};\Q)) @>>> \HH_1(\Mod_{g,1}(L);\HH_1(K_{g,1};\Q))          @>{f_2}>> \Q \\
@.     @VVV                                                   @VV{f_1}V                                                 @VV{\cong}V\\
@.     0                                                 @>>> \HH_1(\Mod_{g,1};\HH_1(\Sigma_{g,1};\Q)) @>>>      \Q
\end{CD}$$
The $\Q$'s on the right hand side of this diagram come from the $H_0$ terms.

The map $f_1$ factors as
$$\HH_1(\Mod_{g,1}(L);\HH_1(K_{g,1};\Q)) \stackrel{f_1'}{\longrightarrow} \HH_1(\Mod_{g,1}(L);\HH_1(\Sigma_{g,1};\Q)) 
\stackrel{f_1''}{\longrightarrow} \HH_1(\Mod_{g,1};\HH_1(\Sigma_{g,1};\Q)).$$
Lemma \ref{lemma:vanishing} says that $f_1'$ is an isomorphism, and Theorem \ref{theorem:hainh1h} together
with Lemma \ref{lemma:transfer} implies
that $f_1''$ is an isomorphism.  We deduce that $f_1$ is an isomorphism.
This implies that $f_2$ is an injection, and hence that
$$\HH_1(\Mod_{g,1}(L);\HH_1(\overline{K}_{g,1};\Q))=0,$$ 
as desired.
\end{proof}

\noindent
We now commence with the proof of Proposition \ref{proposition:weakstability}.

\begin{proof}[{Proof of Proposition \ref{proposition:weakstability}}]
Let $\beta$, $K_{g-1,1}$ and $\overline{K}_{g-1,1}$ be as in Theorem \ref{theorem:birmanexactsequencel2}.  Then Theorem
\ref{theorem:birmanexactsequencel2} together with Lemmas \ref{lemma:weakstabilityvanish} and \ref{lemma:h2decomplemma}
imply that
$$\HH_2(\Mod_{g,\gamma}(L);\Q) \cong \HH_2(\Mod_{g-1,1}(L);\Q) \oplus X,$$
where $X = \Image(\HH_2(\overline{K}_{g-1,1};\Q) \rightarrow \HH_2(\Mod_{g,\gamma}(L);\Q))$.  We must
prove that $X = 0$.  Since $T_{\beta}^L$ is central in both $\overline{K}_{g-1,1}$ and $\Mod_{g,\gamma}(L)$, we
have a commutative diagram of central extensions
$$\begin{CD}
1  @>>> \Z @>>> \overline{K}_{g-1,1}           @>>> K_{g-1,1}                      @>>> 1 \\
@.      @|      @VVV                        @VVV                        @.\\
1  @>>> \Z @>>> \Mod_{g,\gamma}(L) @>>> (\Mod_{g,\gamma}(L))/{\Z} @>>> 1
\end{CD}$$
Here the central $\Z$'s are generated by $T_{\beta}^L$.  Since the map 
$\Mod_{g,\gamma}(L) \rightarrow (\Mod_{g,\gamma}(L))/{\Z}$ is a surjection, we have a surjection
$\HH_1(\Mod_{g,\gamma}(L);\Q) \longrightarrow \HH_1((\Mod_{g,\gamma}(L))/{\Z};\Q)$.
Thus Lemma \ref{lemma:h1modlstab} implies that $\HH_1((\Mod_{g,\gamma}(L))/{\Z};\Q)=0$.
Since $K_{g-1,1}$ is free, the map of Gysin sequences associated to the above commutative diagram of central extensions
contains the commutative diagram of exact sequences
$$\begin{CD}
\HH_1(K_{g-1,1};\Q) @>>> \HH_2(\overline{K}_{g-1,1};\Q) @>>> 0    \\
@VVV                     @VVV                                @VVV      \\
0                   @>>> \HH_2(\Mod_{g,\gamma}(L);\Q)   @>>> \HH_2((\Mod_{g,\gamma}(L))/{\Z};\Q)
\end{CD}$$
An easy diagram chase establishes that the map $\HH_2(\overline{K}_{g-1,1};\Q) \rightarrow \HH_2(\Mod_{g,\gamma}(L);\Q)$
is the zero map, i.e.\ that $X=0$, as desired.
\end{proof}

\noindent
Department of Mathematics; MIT, 2-306 \\
77 Massachusetts Avenue \\
Cambridge, MA 02139-4307 \\
E-mail: {\tt andyp@math.mit.edu}
\medskip


\begin{thebibliography}{}
\begin{small}
\setlength{\itemsep}{1pt}

\bibitem{AtiyahSpin}
M. F. Atiyah, 
Riemann surfaces and spin structures, 
Ann. Sci. \'Ecole Norm. Sup. (4) {\bf 4} (1971), 47--62.

\bibitem{BassMilnorSerre}
H. Bass, J. Milnor\ and\ J.-P. Serre,
Solution of the congruence subgroup problem
for ${\rm SL}\sb{n}\,(n\geq 3)$ and ${\rm Sp}\sb{2n}\,(n\geq 2)$,
Inst. Hautes \'Etudes Sci. Publ. Math. No. 33 (1967), 59--137.

\bibitem{BirmanSiegel}
J. S. Birman, 
On Siegel's modular group, 
Math. Ann. {\bf 191} (1971), 59--68.

\bibitem{BorelDensity}
A. Borel,
Density properties for certain subgroups of semi-simple groups without compact components,
Ann. of Math. (2) {\bf 72} (1960), 179--188.

\bibitem{BorelStability1}
A. Borel, 
Stable real cohomology of arithmetic groups, 
Ann. Sci. \'Ecole Norm. Sup. (4) {\bf 7} (1974), 235--272 (1975).

\bibitem{BorelStability2}
A. Borel, 
Stable real cohomology of arithmetic groups. II, 
in {\it Manifolds and Lie groups (Notre Dame, Ind., 1980)}, 
21--55, Progr. Math., 14, Birkh\"auser, Boston, Mass.

\bibitem{BrownCohomology}
K. S. Brown,
{\it Cohomology of groups},
Corrected reprint of the 1982 original, Springer, New York, 1994.

\bibitem{CharneyCongruence}
R. Charney, 
On the problem of homology stability for congruence subgroups, 
Comm. Algebra {\bf 12} (1984), no.~17-18, 2081--2123.

\bibitem{CharneyVogtmann}
R. Charney, 
A generalization of a theorem of Vogtmann, 
J. Pure Appl. Algebra {\bf 44} (1987), no.~1-3, 107--125.

\bibitem{FarbProblems}
B. Farb, 
Some problems on mapping class groups and moduli space, 
in {\it Problems on mapping class groups and related topics}, 
11--55, Proc. Sympos. Pure Math., 74, Amer. Math. Soc., Providence, RI.

\bibitem{FarbMargalitSurvey}
B. Farb\ and\ D. Margalit,
A primer on mapping class groups,
preprint.

\bibitem{Foisy}
J. Foisy, 
The second homology group of the level $2$ mapping class group and extended Torelli group of an orientable surface, 
Topology {\bf 38} (1999), no.~6, 1175--1207.

\bibitem{GalatiusSpin}
S. Galatius, 
Mod $2$ homology of the stable spin mapping class group, 
Math. Ann. {\bf 334} (2006), no.~2, 439--455. 

\bibitem{HainTorelli}
R. M. Hain, 
Torelli groups and geometry of moduli spaces of curves, 
in {\it Current topics in complex algebraic geometry (Berkeley, CA, 1992/93)}, 
97--143, Cambridge Univ. Press, Cambridge.

\bibitem{HarerH2}
J. Harer,
The second homology group of the mapping class group of an orientable surface,
Invent. Math. {\bf 72} (1983), no.~2, 221--239.

\bibitem{HarerStability}
J. L. Harer, 
Stability of the homology of the mapping class groups of orientable surfaces, 
Ann. of Math. (2) {\bf 121} (1985), no.~2, 215--249.

\bibitem{HarerStableSpin}
J. L. Harer, 
Stability of the homology of the moduli spaces of Riemann surfaces with spin structure, 
Math. Ann. {\bf 287} (1990), no.~2, 323--334. 

\bibitem{HarerH2Spin}
J. L. Harer, 
The rational Picard group of the moduli space of Riemann surfaces with spin structure, 
in {\it Mapping class groups and moduli spaces of Riemann surfaces (G\"ottingen, 1991/Seattle, WA, 1991)}, 
107--136, Contemp. Math., 150, Amer. Math. Soc., Providence, RI.

\bibitem{IvanovBook}
N. V. Ivanov, 
{\it Subgroups of Teichm\"uller modular groups}, 
Translated from the Russian by E. J. F. Primrose and revised by the author, 
Amer. Math. Soc., Providence, RI, 1992.

\bibitem{IvanovComm}
N. V. Ivanov, 
Automorphism of complexes of curves and of Teichm\"uller spaces, 
Internat. Math. Res. Notices {\bf 1997}, no.~14, 651--666.

\bibitem{IvanovSurvey}
N. V. Ivanov, 
Mapping class groups, 
in {\it Handbook of geometric topology}, 
523--633, North-Holland, Amsterdam.

\bibitem{JohnsonFirst}
D. L. Johnson, 
Homeomorphisms of a surface which act trivially on homology, 
Proc. Amer. Math. Soc. {\bf 75} (1979), no.~1, 119--125.

\bibitem{JohnsonSpin}
D. Johnson, 
Spin structures and quadratic forms on surfaces, 
J. London Math. Soc. (2) {\bf 22} (1980), no.~2, 365--373.

\bibitem{JohnsonFinite}
D. Johnson,
The structure of the Torelli group. I. A finite set of generators for ${\cal I}$,
Ann. of Math. (2) {\bf 118} (1983), no.~3, 423--442.

\bibitem{JohnsonAbel}
D. Johnson, 
The structure of the Torelli group. III. The abelianization of $\Torelli$, 
Topology {\bf 24} (1985), no.~2, 127--144.

\bibitem{KnudsonBook}
K. P. Knudson, 
{\it Homology of linear groups}, 
Progr. Math., 193, Birkh\"auser, Basel, 2001.

\bibitem{MadsenWeiss}
I. Madsen\ and\ M. Weiss, 
The stable moduli space of Riemann surfaces: Mumford's conjecture, 
Ann. of Math. (2) {\bf 165} (2007), no.~3, 843--941.

\bibitem{McCarthyCofinite}
J. D. McCarthy, 
On the first cohomology group of cofinite subgroups in surface mapping class groups, 
Topology {\bf 40} (2001), no.~2, 401--418.

\bibitem{MennickeCongruence}
J. Mennicke,
Zur Theorie der Siegelschen Modulgruppe,
Math. Ann. {\bf 159} (1965), 115--129.

\bibitem{PennerProblems}
R. C. Penner, 
Probing mapping class groups using arcs, 
in {\it Problems on mapping class groups and related topics}, 
97--114, Proc. Sympos. Pure Math., 74, Amer. Math. Soc., Providence, RI.

\bibitem{PowellTorelli}
J. Powell, 
Two theorems on the mapping class group of a surface, 
Proc. Amer. Math. Soc. {\bf 68} (1978), no.~3, 347--350.

\bibitem{PutmanCutPaste}
A. Putman, 
Cutting and pasting in the Torelli group, 
Geom. Topol. {\bf 11} (2007), 829--865.

\bibitem{PutmanInfinite}
A. Putman,
An infinite presentation of the Torelli group,
to appear in GAFA.

\bibitem{PutmanH1Level}
A. Putman,
The abelianization of the level $L$ mapping class group,
preprint 2008.

\bibitem{PutmanVanishing}
A. Putman,
Abelian covers of surfaces and the homology of the level $L$ mapping class group,
preprint 2009.

\bibitem{PutmanPicardGroup}
A. Putman,
The Picard group of the moduli space of curves with level structures,
preprint 2009.

\bibitem{RourkeSanderson}
C. P. Rourke\ and\ B. J. Sanderson,
{\it Introduction to piecewise-linear topology},
Reprint, Springer, Berlin, 1982.

\bibitem{Spanier}
E. H. Spanier,
{\it Algebraic topology},
Corrected reprint, Springer, New York, 1981.

\bibitem{TrappLinear}
R. Trapp, 
A linear representation of the mapping class group ${\mathcal M}$ and the theory of winding numbers, 
Topology Appl. {\bf 43} (1992), no.~1, 47--64.

\bibitem{VanDerKallen}
W. van der Kallen, 
Homology stability for linear groups, 
Invent. Math. {\bf 60} (1980), no.~3, 269--295.

\bibitem{ZeemanSpec}
E. C. Zeeman, 
A proof of the comparison theorem for spectral sequences, 
Proc. Cambridge Philos. Soc. {\bf 53} (1957), 57--62.

\bibitem{ZeemanSimp}
E. C. Zeeman,
Relative simplicial approximation,
Proc. Cambridge Philos. Soc. {\bf 60} (1964), 39--43.

\bibitem{ZimmerBook}
R. J. Zimmer, 
{\it Ergodic theory and semisimple groups}, 
Birkh\"auser, Basel, 1984.

\end{small}
\end{thebibliography}
\end{document}